\documentclass[11pt,reqno]{amsart}

\usepackage{graphicx,color,amsmath,amssymb,mathrsfs}
\usepackage{amsfonts,amsthm,epsfig,epstopdf}

\usepackage{pgf,tikz}
\usetikzlibrary{arrows}
\usepackage{subfigure}

\usepackage{booktabs}

\usepackage{url}

\definecolor{grey}{rgb}{.7,.7,.7}

\newtheorem{theorem}{Theorem}[section]
\newtheorem{lemma}[theorem]{Lemma}
\newtheorem{proposition}[theorem]{Proposition}
\newtheorem{corollary}[theorem]{Corollary}
\newtheorem{remark}[theorem]{Remark}
\newtheorem{definition}[theorem]{Definition}
\newtheorem{example}[theorem]{Example}

\newcommand{\DIV}{{\rm div \,}}

\newcommand{\R}{{\mathbb{R}}}
\newcommand{\N}{{\mathbb{N}}}

\newcommand{\e}{{\varepsilon}}
\newcommand{\g}{{\gamma}}
\newcommand{\G}{{\Gamma}}

\newcommand{\pa}{{\partial}}
\newcommand{\C}{{\mathcal{C}}}
\newcommand{\M}{{\mathcal{M}}}

\newcommand{\cI}{{\mathcal{I}}}
\newcommand{\intY}{{Y^{\circ}}}

\newcommand{\dd}{{\rm{dist}}}
\newcommand{\inr}{{\rm{inr}}}
\newcommand{\RR}{{\rm{reach}}}
\newcommand{\Hau}{{\mathcal H}}
\newcommand{\diam}{{\mathrm{diam}}}
\newcommand{\vertiii}[1]{{\left\vert\kern-0.25ex\left\vert\kern-0.25ex\left\vert #1 
    \right\vert\kern-0.25ex\right\vert\kern-0.25ex\right\vert}}

\newcommand{\Om}{\Omega}
\newcommand{\om}{\omega}
\newcommand{\na}{{\nabla}}

\numberwithin{equation}{section}
\begin{document}

\definecolor{qqwuqq}{rgb}{0.,0.39215686274509803,0.}
\definecolor{ttqqqq}{rgb}{0.2,0.,0.}

\title[The Cheeger constant of a Jordan domain without necks]{The Cheeger constant of a Jordan domain without necks}

\author{Gian Paolo Leonardi}
\address{Dipartimento di Scienze Fisiche, Informatiche e Matematiche, Universit{\`a} degli Studi di Modena e Reggio Emilia, via Campi 213/b, I-41125 Modena, Italy}
\email{gianpaolo.leonardi@unimore.it}

\author{Robin Neumayer}
\address{Department of Mathematics, Northwestern University, 2033 Sheridan Road,
Evanston, IL, USA 60208}
\email{neumayer@math.northwestern.edu}

\author{Giorgio Saracco}
\address{Department Mathematik, Universit\"at Erlangen-N\"urnberg, Cauerst. 11, D-91058 Erlangen, Germany}
\email{saracco@math.fau.de}

\thanks{G.P. Leonardi and G. Saracco have been supported by GNAMPA projects: \textit{Problemi isoperimetrici e teoria geometrica della misura in spazi metrici} (2015) and \textit{Variational problems and geometric measure theory in metric spaces} (2016). 
R. Neumayer is supported by the NSF Graduate Research Fellowship under Grant DGE-1110007.}



\begin{abstract}
We show that the maximal Cheeger set of a Jordan domain $\Om$ without necks is the union of all balls of radius $r = h(\Om)^{-1}$ contained in $\Om$. Here, $h(\Om)$ denotes the Cheeger constant of $\Om$, that is, the infimum of the ratio of perimeter over area  among subsets of $\Om$, and a Cheeger set is a set attaining the infimum. The radius $r$ is shown to be the unique number such that the area of the inner parallel set $\Omega^r$ is equal to $\pi r^2$. The proof of the main theorem requires the combination of several intermediate facts, some of which are of interest in their own right. Examples are given demonstrating the generality of the result as well as the sharpness of our assumptions. In particular, as an application of the main theorem, we illustrate how to effectively approximate the Cheeger constant of the Koch snowflake.
\end{abstract}

 \hspace{-2cm}
 {
 \begin{minipage}[t]{0.6\linewidth}
 \begin{scriptsize}
 \vspace{-3cm}
 This is a pre-print of an article published in \emph{Calc. Var. Partial Differential Equations}. The final authenticated version is available online at: http://dx.doi.org/10.1007/s00526-017-1263-0
 \end{scriptsize}
\end{minipage} 
}

\maketitle

\section{Introduction}
Given an open bounded set $\Omega\subset \R^n$ for $n\geq 2,$ we consider the minimization problem
\begin{equation}\label{eqn: Cheeger pb}
h(\Om) = \inf \left\{ \frac{P(F)}{|F|} : F \subseteq \Omega, \ |F|>0 \right\}\,,
\end{equation}
where by $|F|$ and $P(F)$ we denote the volume and the perimeter of $F$ respectively.
This classical isoperimetric-type problem was first considered by Steiner \cite{Steiner1841} and Besicovitch \cite{Besicovitch} in the context of convex subsets of the Euclidean plane; see also \cite{CroftFalconerGuy,SingmasterSouppuris}. 
In \cite{Cheeger70}, Cheeger proved a lower bound on the first eigenvalue of (minus) the Laplace-Beltrami operator on a compact Riemannian manifold in terms of the infimum in the appropriate analogue of \eqref{eqn: Cheeger pb}.
 Since then, the problem \eqref{eqn: Cheeger pb} has been known as the \textit{Cheeger problem} and has appeared in a number of fields like capillarity theory \cite{Giusti78, Finn1984, LS16a}, image processing \cite{BelCasNov2002,AltCasCha2005,CasChaNov2011}, landslide modeling \cite{HILR02,IL05, HIL05}, and fracture mechanics \cite{Kel80}. We refer the reader to the survey paper \cite{Leo15} for a further discussion of applications.

The infimum in \eqref{eqn: Cheeger pb} is readily shown to be a minimum via the direct method in the calculus of variations. Any set attaining the minimum is known as \textit{Cheeger set} of $\Om$, and $h(\Om)$ is called the \textit{Cheeger constant} of $\Om$. Various geometric properties of Cheeger sets can be deduced, two of which we mention now. First, since a Cheeger set $E$ minimizes the perimeter among sets $F \subset \Om$ with $|F|=|E|$, classical regularity results show that $\pa E \cap \Om$ is an analytic hypersurface outside a closed singular set of Hausdorff dimension at most $n-8$, with constant mean curvature equal to $h(\Om)$. In particular, if $n=2$, then $\pa E\cap \Om$ is a union of circular arcs of radius $r=1/h(\Om)$. Second, while uniqueness is false in general, the class of Cheeger sets of $\Om$ is closed under countable unions (see, for instance, \cite[Proposition~2.5, Examples~4.5-4.6]{PraLeo14}), and so one can define the \textit{maximal Cheeger set} of $\Om$ as the union of all its Cheeger sets. 

Computing the Cheeger constant and finding the Cheeger sets of a given domain $\Om$ is a generally difficult problem. Some numerical methods based on duality theory were employed in \cite{LRO05, Buttazzo07, CarComP09} to approximate the maximal Cheeger set of $\Om$. However, apart from a few specific examples, Cheeger sets have been precisely characterized for only two classes of domains: convex planar sets and planar strips. In \cite{Besicovitch}, (see also \cite{SingmasterSouppuris, StredZ97, KawFrid03, Kawohl06}) it is shown that if $\Omega \subset \R^2$ is convex, then it has a unique Cheeger set $E$ given by
\begin{equation}\label{eqn: union}
E = \bigcup_{x \in \Om^{r}} B_r(x), \qquad r = \frac{1}{h(\Om)}\,.
\end{equation}
where $\Om^{r}= \{ x \in \Om : \dd(x,\pa \Om)\ge r\}$ is the so-called \textit{inner Cheeger set}. Moreover, in \cite{Kawohl06}, the {\it inner Cheeger formula} for convex sets is proven. That is, $h(\Om) = 1/r$ where $r>0$ is the unique solution to the equation
\begin{equation}\label{eqn: r}
\pi r^2 = |\Om^r|.
\end{equation}
More recently, further results on the Cheeger problem have been obtained in \cite{KrePra2011} and \cite{PraLeo14} for a class of domains called {\it planar strips}. In particular, in \cite{PraLeo14}, the Cheeger set of a strip is shown to satisfy \eqref{eqn: union} and \eqref{eqn: r} as well. In the same paper, the inner Cheeger formula for a strip of width $2$ and length $L$ is used to provide a first order expansion of the Cheeger constant in terms of $1/L$, as $L\to+\infty$ (see Theorem 3.2 in \cite{PraLeo14}).

\begin{figure}[t]%
\centering
	\subfigure[The Cheeger set $E_{\mathcal W}$ of a bow-tie domain strictly contains the right-hand side of \eqref{eqn: union}.\label{fig:bowtie}]{\includegraphics[scale=1]{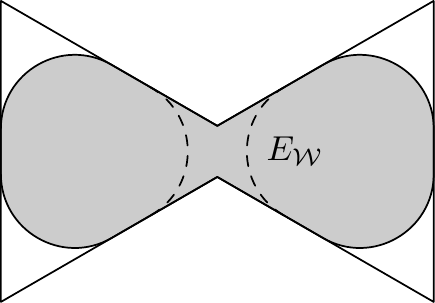}}%
	\hspace{2.5cm}
	\subfigure[The Cheeger set of an unbalanced barbell is strictly contained in the right-hand side of  \eqref{eqn: union}.\label{fig:barbell}]{\includegraphics[width=.28\columnwidth]{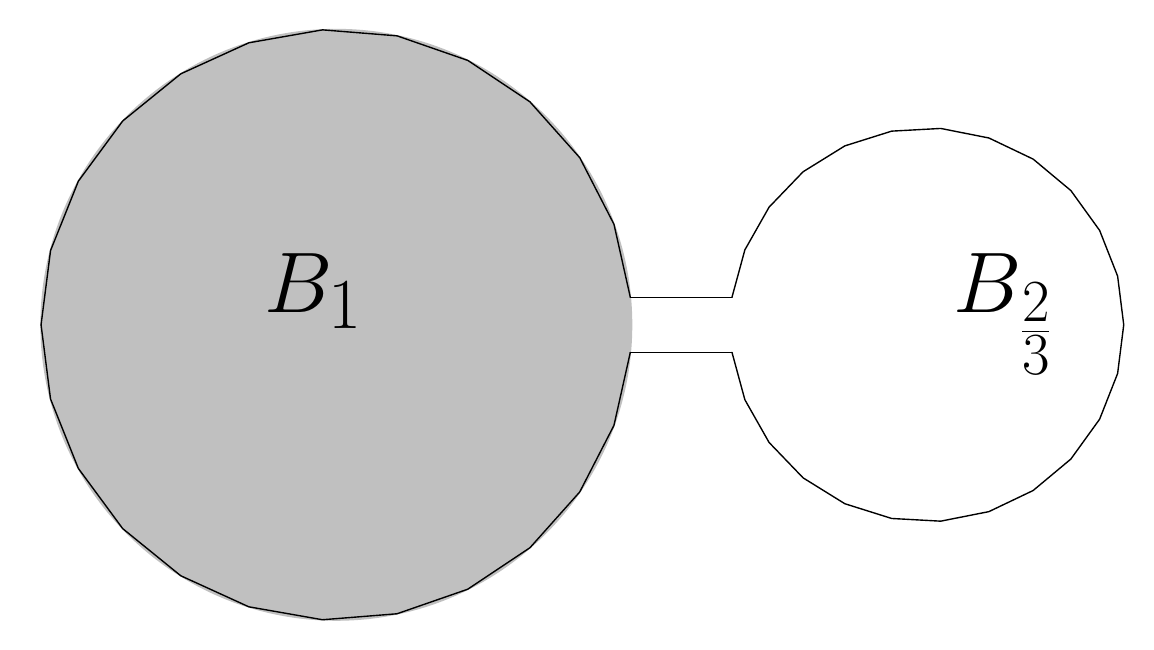}}
\caption{The above figures display different situations where \eqref{eqn: union} does not hold.}
\label{tutta}
\end{figure}

One should not expect a characterization of the type \eqref{eqn: union} to hold in general. Since the constant mean curvature condition forces $\pa E \cap \Om$ to comprise spherical caps only when $n= 2$, it is unsurprising that counterexamples are easily found for $n\geq 3$. But, even in $\R^2$, \eqref{eqn: union} can fail to hold for every Cheeger set of certain domains. For instance, we recall the bow-tie domain $\mathcal{W}$ constructed in 
\cite[Example 4.2]{PraLeo14} and depicted here in Figure~\ref{fig:bowtie}. The Cheeger set $E_{\mathcal W}$ of $\mathcal{W}$ is unique, but $E_{\mathcal W}$ includes the ``neck'' of $\mathcal{W}$, which is not contained in the right-hand side of \eqref{eqn: union}. The opposite situation can also occur. Consider an ``unbalanced'' barbell $\Om$ made of two disks of radii $1$ and $2/3$ at positive distance and connected by a thin tube; see Figure \ref{fig:barbell}. The unique Cheeger set is a small perturbation of the disk of radius $1$ yielding a value of $h(\Omega)$ as close as we wish (depending of the width of the tube) to $2$. Therefore, the union of all disks of radius $1/2$ contained in $\Om$ strictly contains the maximal Cheeger set, as it also includes the entire disk of radius $2/3$.

In general, given a Cheeger set $E$ of $\Om$, it is not even true that every connected component of $\pa E\cap \Om$ is contained in the boundary of a ball of radius $r = h(\Om)^{-1}$ fully contained in $\Om$, as illustrated in the following example.

\begin{example}[The heart domain]\label{ex: heart}
{\rm Given $\theta \in [\pi/4,\pi/2)$, we construct the heart domain $\Om_\theta$ depicted in Figure \ref{fig:heart} in the following way. Let $B\subset \R^2$ be the unit ball of radius $1$ centered at the origin and let $x_\theta$ be the point in $\pa B \cap \{x_2\leq 0\}$ that forms the angle $2\theta $ with $x_0 =(1,0).$ Let $\ell_\theta$ be the tangent line to $B$ at $x_\theta$, and let $A_\theta$ by the region enclosed be the arc of $\pa B$ with angle $2\pi - 2 \theta \geq \pi$, $\ell,$ and $\{x_1 = 1\}.$ Finally, let $\Om_\theta$ be the union of $A_\theta$ with its reflection across $\{x_1 = 1\}$. 

We claim that $\Om_{\pi/4}$ is the unique Cheeger set of itself and $h(\Om_{\pi/4}) = 1 + \frac{3\pi}{3\pi+4}$. 
	Indeed, let us preliminarily compute
\begin{align*}
|\Om_\theta| &= 2\big(\pi - \theta + \tan \theta\big),\\
P(\Om_\theta) &= 4(\pi - \theta) + 2\tan \theta\,.
\end{align*}
We focus on the set $\Om_{\pi/4}$. First, note that the convex hull $S$ of $\Om_{\pi/4}$ is a $2\times 2$ square capped with two half disks of radius $1$. It is easily shown that $S$ is the unique Cheeger set of itself, hence $h(S) = \frac{2\pi + 4}{\pi + 4}$. Since the Cheeger constant is monotonically decreasing with respect to the sets' inclusion, we deduce that 
	\begin{equation}\label{eq:h-heart}
	5/4 < h(S) \le h(\Om_{\pi/4}) \le \frac{P(\Om_{\pi/4})}{|\Om_{\pi/4}|} = 1 + \frac{3\pi}{3\pi+4}  < 2\,.
	\end{equation}
Now we prove that $\Om_{\pi/4}$ is the unique Cheeger set of itself. Indeed, we can argue by contradiction considering an internal arc $\g$ of the boundary of a Cheeger set of $\Om_{\pi/4}$. By well-known properties of Cheeger sets (see Proposition \ref{prop:proprietaCS}), we know that (i) the radius $r$ of $\g$ is in $(1/2, 4/5]$ thanks to \eqref{eq:h-heart}, that (ii) $\g$ meets $\pa \Om_{\pi/4}$ in a tangential way at one endpoint (indeed, $\pa \Om_{\pi/4}$ is regular except three points that are at distance as least $2\sqrt{2}$ from each other), and that (iii) $\g$ is at most a half-circle. Together, (i), (ii) and (iii) imply that one of the endpoints of $\g$ must coincide with the projection of the inner cusp point onto the horizontal segment contained in $\pa\Om_{\pi/4}$, and thus the other one must be the cusp point itself. However this implies that $\g$ is a half-circle and that $r_{\pi/4} = 1/2$, which gives a contradiction with the last inequality in \eqref{eq:h-heart}.

By Theorem 2.7 in \cite{PraLeo14}, we have $h(\Om_{\theta})\to h(\Om_{\pi/4})$ as $\theta\to \pi/4$. Moreover, the Cheeger set $E_\theta$ of $\Om_{\theta}$ is unique and $L^1$-close to $\Om_{\pi/4}$ for $\theta$ close enough to $\pi/4$. Using the symmetry of $\Om_{\theta}$ and arguing as above, one easily shows that $E_\theta$ is as in Figure \ref{fig:heart}.
However, if
\begin{equation}\label{tanheart}
\frac{1}{2} \tan \theta < 1 - \frac{\pi - \theta}{2(\pi - \theta) + \tan \theta}
\end{equation}
then it is impossible to complete the bottom arc of $E_\theta$ with a ball of radius $r = h(\Om_{\theta})^{-1}$ fully contained in $\Om_\theta$. Inequality \eqref{tanheart} is satisfied for $\theta = \pi/4$, and thus for $\theta >\pi/4$ sufficiently close to $\frac{\pi }{4}$ as well.
}

\begin{figure}[t]
\centering
\definecolor{eqeqeq}{rgb}{0.8784313725490196,0.8784313725490196,0.8784313725490196}
\begin{tikzpicture}[line cap=round,line join=round,>=triangle 45,x=1.0cm,y=1.0cm]
\clip(-1.3,-2.3) rectangle (3.3,1.2);
\draw [line width=0.4pt,dash pattern=on 1pt off 1pt,color=eqeqeq,fill=eqeqeq,fill opacity=1.0] (0.,0.) circle (1.cm);
\draw [shift={(0.,0.)},dash pattern=on 1pt off 1pt,fill=black,fill opacity=0.1] (0,0) -- (-129.1528196750553:0.20973517980843692) arc (-129.1528196750553:0.:0.20973517980843692) -- cycle;
\fill[line width=0.pt,color=eqeqeq,fill=eqeqeq,fill opacity=1.0] (-0.631390958857151,-0.7754646717120306) -- (1.,0.) -- (2.631390958857151,-0.7754646717120308) -- (1.8339514176713227,-1.4247477853171082) -- (0.1660485823286771,-1.4247477853171082) -- cycle;
\draw [line width=0.4pt,color=eqeqeq,fill=eqeqeq,fill opacity=1.0] (2.,0.) circle (1.cm);
\draw [dash pattern=on 1pt off 1pt] (-0.631390958857151,-0.7754646717120306)-- (0.,0.);
\draw (-0.631390958857151,-0.7754646717120306)-- (1.,-2.1037592276840296);
\draw [shift={(0.,0.)}] plot[domain=0.:4.029043365615358,variable=\t]({1.*1.*cos(\t r)+0.*1.*sin(\t r)},{0.*1.*cos(\t r)+1.*1.*sin(\t r)});
\draw [shift={(2.,0.)}] plot[domain=0.:4.029043365615358,variable=\t]({-1.*1.*cos(\t r)+0.*1.*sin(\t r)},{0.*1.*cos(\t r)+1.*1.*sin(\t r)});
\draw (2.631390958857151,-0.7754646717120308)-- (1.,-2.1037592276840296);
\draw (1.765047173441928,-0.1030315812691684) node[anchor=north west] {$\Omega_\theta$};
\draw (0.01725400837162019,-0.18692565319254323) node[anchor=north west] {$2\theta$};
\draw [shift={(1.,-0.40050147127606833)}] plot[domain=4.029043365615358:5.395734595154021,variable=\t]({1.*1.3208162169138693*cos(\t r)+0.*1.3208162169138693*sin(\t r)},{0.*1.3208162169138693*cos(\t r)+1.*1.3208162169138693*sin(\t r)});
\draw [shift={(1.,-0.40050147127606833)},line width=0.4pt,dotted,color=eqeqeq,fill=eqeqeq,fill opacity=1.0]  (0,0) --  plot[domain=4.029043365615358:5.395734595154021,variable=\t]({1.*1.3208162169138693*cos(\t r)+0.*1.3208162169138693*sin(\t r)},{0.*1.3208162169138693*cos(\t r)+1.*1.3208162169138693*sin(\t r)}) -- cycle ;
\draw [dash pattern=on 1pt off 1pt] (1.,0.)-- (0.,0.);
\draw [shift={(2.,0.)}] plot[domain=0.:4.029043365615358,variable=\t]({-1.*1.*cos(\t r)+0.*1.*sin(\t r)},{0.*1.*cos(\t r)+1.*1.*sin(\t r)});
\draw (-0.631390958857151,-0.7754646717120306)-- (1.,-2.103759227684029);
\draw (2.631390958857151,-0.7754646717120308)-- (1.,-2.1037592276840296);
\draw [fill=black] (-0.631390958857151,-0.7754646717120306) circle (0.5pt);
\draw[color=black] (-0.9,-0.9) node {$x_\theta$};
\end{tikzpicture}
\caption{Heart domain $\Om_{\theta}$ \label{fig:heart}}
\end{figure}
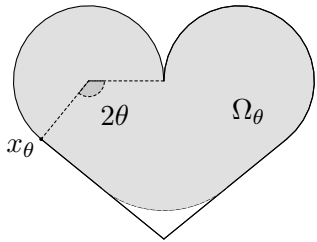
\end{example}

 The main goal of this work is to show that, for a broad class of domains in $\R^2$ that includes both convex sets and strips, the maximal Cheeger set is given by \eqref{eqn: union} and that the Cheeger constant can be determined as the inverse of the unique solution of \eqref{eqn: r}. The examples above will demonstrate that this class of domains is essentially optimal. 
To this end, in the following definition, we introduce the concept of a \textit{set with no necks of radius $r$}.
\begin{definition}\label{def:noneck} A set $\Om$ has no necks of radius $r$ if the following condition holds. If $B_r(x_0)$ and $B_r(x_1)$ are two balls of radius $r$ contained in $\Om$, then there exists a continuous curve $\g : [0,1] \to \Om$ such that 
\begin{align*}
\g(0) = x_0 ,\qquad   \g(1) = x_1, \qquad \text{  and} \qquad B_r(\g(t) )  \subset \Om \text{ for all } t\in [0,1].
 \end{align*}
\end{definition}

\begin{remark}\label{rmk:nobott}
\rm We notice for future reference that the property of having no necks of radius $r$ can be equivalently stated as the path-connectedness of the \textit{inner parallel set} 
\[
\Om^{r} = \{x\in \Om:\ \dd(x,\pa\Om)\ge r\}\,.
\]
\end{remark}

Our main result is the following.

\begin{theorem}\label{theorem:Cheeger formula}
Suppose $\Om \subset \R^2$ satisfies the following two properties:
\begin{align}\label{T}
\tag{T} &\Om\ \text{is a Jordan domain and $|\pa \Om|=0$,}  \\
\tag{NB}\label{NB}  & \Om\ \text{has no necks of radius } r = 1/h(\Om).
\end{align}
Then the maximal Cheeger set $E$ of $\Om$ is given by \eqref{eqn: union} and the inner Cheeger formula \eqref{eqn: r} holds.
\end{theorem}

Uniqueness of Cheeger sets may fail even under the assumptions \eqref{T} and \eqref{NB}, as is seen, for instance, in the Pinocchio set of \cite[Example 4.6]{PraLeo14}). However, in  Remark~\ref{rem:uniqueness}, we prove that if additionally $\overline{\text{\rm int}(\Omega^r)} = \Omega^r$, then the Cheeger set of $\Om$ is unique.

Let us make some comments on the assumptions we make on $\Om$. The assumption \eqref{NB} is essentially necessary. Indeed, if we remove this assumption, then \eqref{eqn: union} and \eqref{eqn: r} may fail, as we have seen in the examples of the bow-tie domain $\mathcal W$, the unbalanced barbell $B_{1, 2/3}$, and the heart domain $\Omega_{\theta}$ shown in Figures \ref{fig:bowtie}, \ref{fig:barbell} and \ref{fig:heart}. The assumption  \eqref{T} implies that any closed loop in $\Omega$ is contractible, which is necessary for \eqref{eqn: union} to hold. Indeed, consider a domain $\Om$ given by a disk $D$ minus a small disk $H$ near $\pa D$, so that $\Om$ is not simply connected.  If the radius of $H$ is small enough, one can show that $\Om$ coincides with its own Cheeger set and properties \eqref{eqn: union} and \eqref{eqn: r} are not satisfied. 
We could relax assumption \eqref{T} slightly by allowing $\Omega$ to be a finite union of Jordan domains. But, \eqref{NB} would force the maximal Cheeger set of $\Omega$ to be contained in a single connected component of $\Omega$, so this relaxation would be inconsequential. It is not clear whether the hypothesis $|\pa \Om| =0$ in \eqref{T} is necessary for Theorem \ref{theorem:Cheeger formula} to hold, but we use it in an essential way in our proof (see Proposition \ref{prop:sconnectednessE}). It is worth recalling the so-called Osgood-Knopp's curve bounding a Jordan domain and having positive Lebesgue measure; see \cite{Sagan}.

Notice that Theorem~\ref{theorem:Cheeger formula} requires the a priori knowledge of the value of $h(\Om)$ to verify hypothesis \eqref{NB}. However, one can replace hypothesis \eqref{NB} by another one slightly stronger but independent of $h(\Om)$. The idea is simple: instead of requiring the no-neck property exactly for the (unknown) radius $r=1/h(\Om)$, we can ask a more restrictive but easier-to-check assumption \eqref{NB'} involving a lower and an upper bound on $r$.  In the following corollary we choose as a competitor the biggest ball that can be fit in $\Omega$, thus yielding a lower bound for $r$, while an upper bound is obtained by applying the isoperimetric inequality. It is, of course, clear that we could have considered other competitors to obtain the lower bound. 
\begin{corollary}\label{cor: cor2}
Suppose $\Om \subset \R^2$ satisfies \eqref{T} and 
\begin{align}\tag{NB$'$}\label{NB'}
\text{$\Om$ has no necks of radius $s$ for all}\ \ \frac{\inr(\Om)}{2} \le s  \le \frac 12 \left(\frac{|\Om|}{\pi}\right)^{1/2}\,,
\end{align}
where $\inr(\Om)$ denotes the inradius of $\Om$. Then the maximal Cheeger set $E$ of $\Om$ is given by \eqref{eqn: union} and $r$ is the unique value for which \eqref{eqn: r} is satisfied.
\end{corollary}

Finally, as an application of Corollary~\ref{cor: cor2}, we give explicit formulas for the computation of the Cheeger constant of a Koch snowflake $K$; see Section \ref{sec:Koch}. As these formulas are not algebraically solvable, we give an approximation that relies on an error estimate between $h(K)$ and $h(K_n)$, where $K_n$ denotes the $n$-th step of the construction of $K$ (see Lemma \ref{lem:approx}). Specifically, we compute the exact value of $h(K_n)$, for $n=2,3,4$. Then,  for larger $n$ we describe the procedure to obtain a numerical approximation of $h(K_n)$. For the sake of simplicity, we analyze with full details the case $n=5$ and finally we show that $h(K) \approx 1.89124548$.
\\

Let us briefly discuss the method of proof for Theorem~\ref{theorem:Cheeger formula}, which consists of three main steps. First, we show that a Cheeger set $E$ of $\Om$ contains a ball of radius $r=1/h(\Om)$. Since the curvature of $E$ is bounded from above by $h(\Om)$ in the viscosity sense (see Definition~\ref{def:BCVS} and Lemma~\ref{lem:DCVC}), this follows from the next theorem.
\begin{theorem}\label{theorem:Pestov and Ionin}
Suppose $E \subset \R^2$ is a Jordan domain with curvature bounded from above by $h>0$ in the viscosity sense. Then $E$ contains a ball of radius $1/h$.
\end{theorem}
Theorem~\ref{theorem:Pestov and Ionin} generalizes a result known for domains with $C^2$ boundary (see \cite{PeIo59, HowTre95}, cf. \cite[30.4.1]{BZ88}). We could not find a proof of this fact when the curvature is only known to be bounded in a weak sense, so we prove it here. Our proof of Theorem~\ref{theorem:Pestov and Ionin} partially follows the one given in \cite[Section 2]{HowTre95}. However, some key notions and steps in the original proof must be modified to allow for the lack of regularity of $\pa E$ (we refer in particular to the notions of cut point and of focal point of $E$, as well as to the structure properties of the cut locus; see Section \ref{section:3}). 
Apart from being a key tool in the proof of Theorem~\ref{theorem:Cheeger formula}, Theorem~\ref{theorem:Pestov and Ionin} seems interesting in its own right. In particular, it implies the so-called {\it Almgren's isoperimetric principle} in the planar case; see \cite{Almgren86, KrummelMaggi}.

The second step of the proof of Theorem~\ref{theorem:Cheeger formula} is to show that $E$ contains the union of all balls of radius $r=1/h(\Om)$ contained in $\Om$.  We have at our disposal the following \textit{rolling ball lemma} (see \cite[Lemma 2.12]{PraLeo14}): 
\begin{lemma}[Rolling ball]\label{lem:rollingball}
	If $E$ is a maximal Cheeger set of $\Om$ that contains $B_r(x_0)$ for $r = h(\Om)^{-1}$ and for some $x_{0}$, then $E$ contains any ball of the same radius that can be reached by rolling $B_r(x_0)$ inside of $\Om$. 
\end{lemma}
In the original version of the lemma proven in \cite{PraLeo14}, the curve of centers along which the ball is rolled was required to be of class $C^{1,1}$ and to have curvature bounded by $1/r$, and the proof used this assumption in an essential way. In order to relax this requirement, we show that the same property can be deduced when the curve of centers is only continuous. To this aim, we prove the following theorem, which essentially states that every $r$-tubular neighborhood of a continuous curve $\g\subset\R^2$ contains a curve $\sigma$ admitting a regular parametrization of class $C^{1,1}$, with same endpoints as $\g$ and $L^\infty$-norm of the curvature bounded by $1/r$, such that the $r$-neighborhood of $\sigma$ is contained in the $r$-neighborhood of $\g$. 
\begin{theorem}\label{lem:eraser}
Given any continuous curve $\g:[0,1]\to\R^{2}$ and $r>0$, and defining $U_{\g,r} = \bigcup_{t\in [0,1]} B_{r}(\g(t))$, there exists a curve $\sigma:[0,1]\to \R^{2}$ of class $C^{1,1}$ with curvature bounded by $1/r$ such that 
\[
\sigma(0) = \g(0),\quad \sigma(1) = \g(1),\quad U_{\sigma,r} \subset U_{\g,r}\,.
\]
\end{theorem}
The second step is thus shown by coupling Lemma \ref{lem:rollingball} and Theorem~\ref{lem:eraser} with the no-necks assumption \eqref{NB}.

In the final step of the proof, we show that $E$ is in fact equal to the union of all balls of radius $1/h(\Om)$ contained in $\Om$. 
The properties of cut and focal points of $E$ are also the key in proving this step. 
 Essentially, if $G$ denotes the complement in $E$ of this union of balls, and if one assumes that $G$ is nonempty, then a contradiction is reached as the cut locus must have a nonempty intersection with $G$ and, necessarily, this intersection must be made of cut points whose distance from $\pa E$ is strictly less than $r$. Finally, to prove the inner Cheeger formula \eqref{eqn: r}, we make use of the Steiner formulas.
\\

The paper is organized in the following way. Section~\ref{sec: preliminaries} includes preliminary results that will be needed in the paper. We prove Theorems~\ref{theorem:Pestov and Ionin} and \ref{lem:eraser} in Sections~\ref{section:3} and \ref{sec: eraser} respectively. Section~\ref{sec: proof of main} is dedicated to showing Theorem~\ref{theorem:Cheeger formula} and Corollary~\ref{cor: cor2}. Finally, we compute the Cheeger constant of the Koch snowflake in Section~\ref{sec:Koch}.

\section{Preliminaries}\label{sec: preliminaries}

\subsection{Curvature bounds in the variational and viscosity sense}

\begin{definition}[Curvature bounded from above in the variational sense] Given $h \in L^1_{loc}(\R^{2})$, a set of locally finite perimeter $E\subset \R^2$ has variational curvature bounded above by $h$ if there exists $r_0>0$ such that, for all $0<r<r_{0}$, $x_0\in \R^2$, and $F \subset E$ with $F\Delta E \subset \subset B_r(x_0)$, we have
\[
 P(E; B_r(x_0 )) \leq P(F ; B_r(x_0)) + \int_{E\setminus F} h(x) \, dx.
\] 
\end{definition}

\begin{lemma}\label{lemma:curvature bound}
A Cheeger set $E$ of $\Om$ has variational curvature bounded above by $h(\Om)$.
\end{lemma}
\begin{proof}
A Cheeger set $E$ minimizes the energy 
\[
I(E) = P(E) - h(\Om) |E|
\]
among all sets $F \subset E$. Rearranging the inequality $I(E) \leq I(F),$ we find
\[
P(E) \leq P(F) + h(\Om)(|E| - |F|) = P(F) + \int_{E\setminus F} h(\Om) \, dx.
\]
\end{proof} 

Up to a modification by a set of Lebesgue measure zero, we may assume that $\overline {\pa^*E}= \pa E$ for a set of finite perimeter $E$. We will make this assumption throughout the paper. 
\begin{definition}[Curvature bounded from above in the viscosity sense]\label{def:BCVS} Given a constant $h>0$, a set $E\subset \R^2$ has curvature bounded above by $h$ at $x \in \pa E$ in the viscosity sense if the following holds. Suppose $A$ is a circular arc that locally touches $E$ from outside at $x$, that is, $x \in \pa E$ is contained in the relative interior of $A\subset \pa B_{r}(y)$ and, for some $\e>0$, one has $E\cap B_\e(x) \subset \overline{B}_r(y)$. Then $A$ is an arc of curvature at most $h$, that is, $r\geq 1/h$.
\end{definition}

\begin{lemma}\label{lem:DCVC}
If $E$ has variational mean curvature bounded from above by a constant $h>0$, then $E$ has curvature bounded from above by $h$ in the viscosity sense at every point $x \in \pa E$. 
\end{lemma}
\begin{proof}
Up to rescaling we may assume that $h=1$, i.e.,
\begin{equation}\label{eq:distrimin}
P(E) - |E| \le P(F) - |F|\qquad \text{for all $F\subset E$.}
\end{equation}
Arguing by contradiction, we assume that there exist $x_{0},y_{0}\in \R^{2}$, $\e>0$ and $0<r<1$, such that $x_{0}\in \pa E \cap \pa B_{r}(y_{0})$ and $\pa E\cap B_{\e}(x_{0})\setminus \{x_{0}\} \subset B_{r}(y_{0})$. Up to an isometry, we have $y_{0}=(0,0)$ and $x_{0} = (0,r)$. In the infinite strip $S = (-1,1)\times \R$ we consider the unit vector field $g(x_{1},x_{2}) = (x_{1},\sqrt{1-x_{1}^{2}})$ with divergence constantly equal to one. Consider the one-parameter family of unit half-circles
\[
A_{t} = \left\{(x_{1},x_{2}):\ |x_{1}|<1,\ x_{2} = t+\sqrt{1 - x_{1}^{2}}\right\},\quad t\in \R\,,
\]
which foliates $S$ (notice that $g$ is normal to $A_{t}$ for all $t$). Let $t_{0}$ be such that $x_{0}\in A_{t_{0}}$. For any $t\in (t_{0}-\e,t_{0})$ we set
\[
E_{t} = \big(E\setminus B_{\e}(x_{0})\big) \cup \big( E \cap B_{\e}(x_{0})\cap \left\{x\in S:\ x\text{ lies below }A_{t}\right\}\big)\,.
\]
Note that $E_{t_{0}} = E$ because $A_{t_{0}}$ is outside $B_{r}$, while in general $E_{t}\subset E$. By the divergence theorem, for almost every $t\in (t_{0}-\e,t_{0})$ we get
\[
|E\setminus E_{t}| = \int_{E\setminus E_{t}} \DIV g = \int_{\pa_{t,+}E\cap B_{\e}(x_{0})} g\cdot \nu_{E} - \Hau^{1}(E\cap A_{t}\cap B_{\e}(x_{0}))\,,
\]
where $\pa_{t,+}E$ denotes the set of points in $\pa E$ that belong to $S$ and stay below $A_{t}$. Since $|g| = 1$, this last computation shows in particular that
\begin{equation}\label{eq:stimastretta}
|E\setminus E_{t}| < \Hau^{1}(\pa_{t,+}E\cap B_{\e}(x_{0})) - \Hau^{1}(E\cap A_{t}\cap B_{\e}(x_{0}))\,.
\end{equation}
Notice that the previous inequality is always strict since otherwise each connected component of $\pa_{t,+}E$ is contained in some arc $A_{t'}$, but then necessarily $t'=t$ and thus $x_{0}$ cannot be a point of $\pa E$, that is, a contradiction. Now we set 
\[
a = \Hau^{1}(E\cap A_{t}\cap B_{\e}(x_{0})),\qquad b = \Hau^{1}(\pa_{t,+}E\cap B_{\e}(x_{0}))\,,
\]
and estimate $P(E_{t})$ and $|E_{t}|$ using \eqref{eq:stimastretta}:
\[
|E_{t}| = |E| - |E\setminus E_{t}| > |E| - (b-a)\,, \qquad 
P(E_{t}) = P(E) - (b - a)\,.
\]
Together, these imply that
\[
P(E_{t}) - |E_{t}| < 
P(E) - |E|\,, 
\]
which contradicts \eqref{eq:distrimin}. 
\end{proof}

\subsection{Regularity, area bounds, and topological properties of planar Cheeger sets}
$ $

We start recalling some well-known facts about planar Cheeger sets (see, for instance, \cite{PraLeo14}).
\begin{proposition}\label{prop:proprietaCS}
	Let $\Omega \subset \R^2$ be open, bounded and connected, and let $E$ be a Cheeger set of $\Omega$. Then the following properties hold:
	\begin{itemize}
		\item[(i)] the free boundary of $E$, i.e. $\pa E \cap \Omega$, is analytic and has constant curvature equal to $h(\Omega)$, hence $\pa E \cap \Omega$ is a union of arcs of circle of radius $r=h^{-1}(\Omega)$;
		\item[(ii)] any arc in $\pa E \cap \Omega$ cannot be longer than $\pi r$;
		\item[(iii)] if an arc in $\pa E \cap \Omega$ meets $\pa \Omega$ at a regular point of $\pa \Omega$, then they must meet tangentially;
		\item[(iv)] the area of $E$ is bounded from below as follows:
		\begin{equation}\label{eq: bound on volume}
		|E| \geq \pi \left (\frac{2}{h(\Omega)} \right)^2\,.
		\end{equation}
	\end{itemize}
\end{proposition}

\begin{definition}\label{def:lochomeo}
Let $\Omega\subset \R^{2}$ be an open bounded set. We say that $\pa \Omega$ is locally homeomorphic to an interval if there exist $r_{0}>0$ and a modulus of continuity $\omega_{0}$ such that for any $y_{1},y_{2}\in \pa \Omega$ with $|y_{1}-y_{2}|<r_{0}$, one can find an open set $U$ containing $y_{1},y_{2}$ with $\diam(U)\le \omega_{0}(|y_{1}-y_{2}|)$ such that $\pa \Omega \cap U$ is homeomorphic to an open interval.
\end{definition}

The following proposition shows that the boundary of every Jordan domain is locally homeomorphic to an interval in the sense of Definition \ref{def:lochomeo}.
\begin{proposition}\label{prop:Jordan}
If $\Omega$ is a Jordan domain, then $\pa \Omega$ is locally homeomorphic to an interval.
\end{proposition}
\begin{proof}
By definition of Jordan curve, there exists a homeomorphism $\gamma : \pa B_{1} \to \pa \Omega$. By the Jordan-Schoenflies theorem, this can be extended to a homeomorphism $\Gamma : \R^2 \to \R^2$ such that $\Gamma (\pa B_{1}) = \pa \Omega$ and $\Gamma(B_1) = \Omega$. The restriction of $\Gamma$ to $B_2$ is uniformly continuous, hence there exists a modulus of continuity $\eta$ such that the diameter of $\Gamma(B_{t}(x))$ is bounded by $\eta(t/2)$, for all $B_{t}(x)\subset B_{2}$. One can then choose $t_{0}>0$ and a finite covering $\{B_{t_{0}}(x_{1}),\dots,B_{t_{0}}(x_{N})\}$ of $\pa B_{1}$, with $x_{j}\in \pa B_{1}$ for all $j$, such that $B_{t}(x)\cap \pa B_{1}$ is homeomorphic to an open interval for all $x\in \pa B_{1}$ and $0<t<2t_{0}$. 
Note that the continuous function $\delta(z) = \min \{|z-x_{j}|:\ j=1,\dots,N\}$ has maximum value $t_0 -\e$ on $\pa B_1$ for some $\e>0$. So, if $|z_1 - z_2| <\e$, then $z_1, z_2 \in B_{t_0}(x_j)$ for some $j$. The uniform continuity of $\G^{-1}$ on $\G(B_2)$ then ensures that there exists $r>0$ such that if $|y_{1}-y_{2}|<r$ for $y_{1},y_{2}\in \pa \Omega$, then $y_{1},y_{2}\in \Gamma(B_{t_{0}}(x_{j}))$ for some $j\in \{1,\dots,N\}$. Let $r_0$ be the supremum of all such $r>0$.
Then, for $d< r_0$, setting
\begin{align*}
\omega(d) &= \sup \Big\{\eta(t):\ t = |\Gamma^{-1}(y_{1}) - \Gamma^{-1}(y_{2})|,\ y_{i}\in \pa \Omega,\ d=|y_{1}-y_{2}|\Big\}\,,\\
t_{1,2} &= |\Gamma^{-1}(y_{1}) - \Gamma^{-1}(y_{2})|,\qquad x_{1} = \Gamma^{-1}(y_{1}),\qquad U = \Gamma(B_{2t_{1,2}}(x_{1})), 
\end{align*}
one has that $y_{1},y_{2}\in U$, $\diam(U) \le \om(|y_{1}-y_{2}|)$, and that
$\pa \Omega \cap U = \Gamma(B_{2|y_{1}-y_{2}|}(y_{1}) \cap \pa B_{1})$ is homeomorphic to an open interval, as desired.
\end{proof}

\begin{remark}\rm
We observe that the converse of Proposition \ref{prop:Jordan} is also true, namely, that every bounded open set whose boundary is locally homeomorphic to an open interval is enclosed by a finite family of pairwise disjoint Jordan curves.
\end{remark}

Before going on we recall the definition of \textit{P-indecomposability} (see, for instance, \cite{ACMM}). We say that a set of finite perimeter $E$ is P-indecomposable if, whenever $F$ and $G$ are measurable sets such that $E = F\cup G$ and $P(E) = P(F)+P(G)$, then necessarily $|F|\cdot |G| =0$. It is possible to prove that every set of finite perimeter is a countable union of mutually disjoint, P-indecomposable sets. Note that a $P$-indecomposable component of a Cheeger set is itself a Cheeger set.

\begin{proposition} 
Suppose $\Om \subset \R^2$ is an open bounded set such that $\pa \Omega$ is a finite union of pairwise disjoint Jordan curves. Then, for any Cheeger set $E$ of $\Omega$, $\pa E$ is a finite union of pairwise disjoint Jordan curves.
\end{proposition}
\begin{proof}
	Let $E$ be a Cheeger set in $\Omega$ and let $x \in \pa E$. It suffices to find an open neighborhood $U$ of $x$ and a continuous injective curve $\g$ such that $\pa E\cap U = \text{Im}(\g)$. The free boundary $\pa E \cap \Omega$ is analytic by Proposition~\ref{prop:proprietaCS}(i). Now, let $x\in \pa E \cap \pa \Omega$. By Proposition \ref{prop:Jordan}, there are an open neighborhood $U$ of $x$ with arbitrarily small diameter and a simple parametric curve $\sigma :I \to \R^2$, with $I = (-1,1)$, such that $\sigma(I) = \pa \Om \cap U$ and $\sigma(0) = x$. So, if $x \in \pa E \cap \pa \Omega \setminus \overline{\pa E\cap \Om}$, then up to a reparametrization, we can take $\g(t) = \sigma(t)$ for $t$ close to $0$. This proves the claim in a possibly smaller neighborhood $\widetilde U\subset U$ of $x$. We now consider the case $x \in \pa E \cap \pa \Omega \cap (\overline{\pa E\cap \Om})$, which we split in two steps. 
	
	\emph{Step one.} Up to choosing a smaller neighborhood $U$ of $x$, the connected components of $\pa E \cap \Om\cap U$ are circular arcs with radius $r = 1/h(\Om)$ that satisfy the following properties:
	\begin{itemize}
		\item[(i)] if $\alpha$ is a connected component of $\pa E \cap \Om\cap U$ such that $\dd(x,\alpha)\le  \dd(x,U^{c})/2$, then both endpoints of $\alpha$ belong to $\sigma(I)$, up to a possible exception of a finite number of arcs with no more than one endpoint on $\sigma(I)$;
		\item[(ii)] if $\alpha\subset \pa E \cap \Om\cap U$ has both endpoints on $\sigma(I)$, and if $a_{i}= \sigma(t_{i})$ for $i=1,2$ denote the endpoints of $\alpha$, then either $t_{1},t_{2}\le 0$ or $t_{1},t_{2}\ge 0$ (in other words, the two endpoints lie on the same ``side'' of $\sigma$ with respect to $\sigma(0)=x$);
		\item[(iii)] let $\alpha,\beta\subset \pa E \cap \Om\cap U$ denote two distinct arcs with all endpoints on $\sigma(I)$ and let $a_{i}= \sigma(t_{i})$ and $b_{i}= \sigma(s_{i})$, $i=1,2$, respectively denote the endpoints of $\alpha$ and $\beta$. Then, assuming without loss of generality that $t_{1}<t_{2}$ and $s_{1}<s_{2}$, we either have $t_{2}\le s_{1}$ or $s_{2}\le t_{1}$.
	\end{itemize}
	Indeed, if an endpoint of an arc $\alpha \subset \pa E \cap \Om\cap U$ does not lie in $\sigma(I)$, then it must lie in $U^c$. Hence, if $\alpha$ has at least an endpoint lying in $U^c$ and $B_r(x)\cap \alpha$ is nonempty, where $r = \dd(x, U^c)/2$, then the length of $\alpha$ is at least $r$. As $E$ has finite perimeter, there can only be finitely many such arcs; up to decreasing the diameter of $U$, this proves (i). 
	To prove (ii), assume that there exists $\alpha \subset \pa E\cap U\cap \Om$ with endpoints $a_{1} = \sigma(t_{1})$ and $a_{2}=\sigma(t_{2})$ such that $t_{1}<0<t_{2}$. Concatenating the curves $\alpha$ and $\sigma([t_{1},t_{2}])$, we obtain a Jordan curve $\gamma$ that encloses a domain $D$. Since $x$ belongs to the closure of $\pa E\cap \Om$ but does not belong to $\alpha$, it follows that $D\cap E$ has positive measure. On the other hand, $D\cap E$ is also a P-indecomposable component of $E$, thus it is a Cheeger set of $\Om$. Since the Jordan curve $\gamma$ is contained in $U$ and $U$ has an arbitrarily small diameter, we reach a contradiction with the volume lower bound \eqref{eq: bound on volume}. Finally, (iii) is shown by arguing as in the proof of (ii).
	
	\emph{Step two.}  We first show that $x$ cannot be the endpoint of $m\ge 3$ connected components $\alpha_{1},\dots,\alpha_{m}$ of $\pa E\cap \Om\cap U$. Assume for contradiction that $m\ge 3$ such arcs meet at $x$, choose $t>0$ such that $B_{t}(x)\subset U$, and denote by $q_{j}$ the intersection of $\alpha_{j}$ with $\pa B_{t}(x)$, for $j=1,\dots, m$. Assume without loss of generality that the $q_{j}$s are ordered with respect to the standard, positive orientation of $\pa B_{t}(x)$, and that the arc of $\pa B_{t}(x)$ between $q_{1}$ and $q_{2}$ is contained in $E$; this can be always guaranteed up to reversing the orientation and consistently relabeling the points. Let $\theta_{12}$ be the angle of the sector containing $\alpha_1,\alpha_{2}$ spanned by the half-tangents to $\alpha_{1}$ and $\alpha_{2}$ at $x$. Observe that $\theta_{12}\ge \pi$, otherwise we could ``cut the angle'' to produce an admissible variation of $E$ that would decrease the quotient $P(E)/|E|$. Now, the arc $\alpha_{3}$ must be tangent to $\alpha_{2}$ at $x$, otherwise we could again produce a better competitor by perturbing the Cheeger set by cutting the angle $\theta_{31}$ formed by the half-tangents to $\alpha_{3}$ and $\alpha_{1}$. However, we can now shortcut $\alpha_{2}\cup \alpha_{3}$ near $x$ as depicted in Figure~\ref{fig:shortcut}, producing a competitor with reduced perimeter and increased area.		
		
		Thus, either $x$ is the endpoint of exactly two arcs or it is the endpoint of at most one arc. We use the properties proved in Step one to provide a suitable local parametrization of the boundary.
	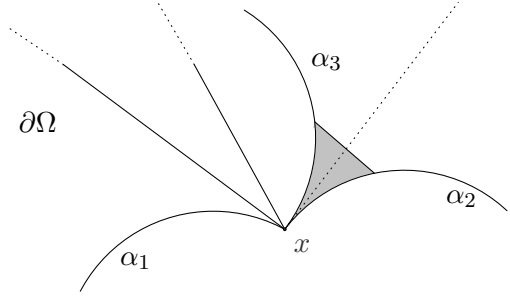
\begin{figure}[t]
		\centering
		\definecolor{ffffff}{rgb}{1.,1.,1.}
		\definecolor{uuuuuu}{rgb}{0.2,0.2,0.2}
		\begin{tikzpicture}[line cap=round,line join=round,>=triangle 45,x=2.0cm,y=2.0cm]
		\clip(-4.1,3.) rectangle (-0.6,5.3);
		\fill[line width=0.pt,color=uuuuuu,fill=uuuuuu,fill opacity=0.3] (-1.96500230400484,4.3743557762752605) -- (-2.1645938890165213,3.6597664631250506) -- (-1.5686820167356883,4.032713621858969) -- cycle;
		\draw [shift={(-1.3699778971494405,3.0526540966976374)},line width=0.4pt,color=ffffff,fill=ffffff,fill opacity=1.0]  (0,0) --  plot[domain=1.7708318232776277:2.4891711284700193,variable=\t]({1.*1.*cos(\t r)+0.*1.*sin(\t r)},{0.*1.*cos(\t r)+1.*1.*sin(\t r)}) -- cycle ;
		\draw [shift={(-2.959209880883603,4.266878829552464)},line width=0.4pt,color=ffffff,fill=ffffff,fill opacity=1.0]  (0,0) --  plot[domain=1.7708318232776277:2.4891711284700193,variable=\t]({-0.26282914906181154*1.*cos(\t r)+0.9648423904469807*1.*sin(\t r)},{0.9648423904469807*1.*cos(\t r)+0.26282914906181154*1.*sin(\t r)}) -- cycle ;
		\draw [shift={(-1.3699778971494405,3.0526540966976374)}] plot[domain=0.81939866086987:2.4891711284700193,variable=\t]({1.*1.*cos(\t r)+0.*1.*sin(\t r)},{0.*1.*cos(\t r)+1.*1.*sin(\t r)});
		\draw [shift={(-2.64,2.78)}] plot[domain=1.075370740362767:2.655388642900473,variable=\t]({1.*1.*cos(\t r)+0.*1.*sin(\t r)},{0.*1.*cos(\t r)+1.*1.*sin(\t r)});
		\draw [shift={(-2.959209880883605,4.266878829552466)}] plot[domain=0.81939866086987:2.4891711284700193,variable=\t]({-0.2628291490618113*1.*cos(\t r)+0.9648423904469808*1.*sin(\t r)},{0.9648423904469808*1.*cos(\t r)+0.2628291490618113*1.*sin(\t r)});
		\draw [dotted] (-2.1645938890165213,3.6597664631250506)-- (-1.0115792490379993,5.168883922180038);
		\draw (-1.96500230400484,4.3743557762752605)-- (-1.5686820167356883,4.032713621858969);
		\draw [color=black] (-2.1645938890165213,3.6597664631250506)-- (-3.634532008815683,4.753955761237582);
		\draw [color=black] (-2.1645938890165213,3.6597664631250506)-- (-2.7657865819204965,4.753955761237582);
		\draw [color=black](-4,4.507029667664105) node[anchor=north west] {$\partial \Omega$};
		\draw [dotted,color=black] (-2.7657865819204965,4.753955761237582)-- (-3.041944068615655,5.225724801008479);
		\draw [dotted,color=black] (-3.6345320088156825,4.753955761237581)-- (-4.,5.);
		\draw [fill=uuuuuu] (-2.1645938890165213,3.6597664631250506) circle (0.5pt);
		\draw[color=uuuuuu] (-2.05,3.55) node {$x$};
		\draw[color=black] (-0.9893884815480305,3.8620083763870525) node {$\alpha_2$};
		\draw[color=black] (-3.156576014092701,3.451522258043791) node {$\alpha_1$};
		\draw[color=black] (-1.8861794884196807,4.772626669954656) node {$\alpha_3$};
		\end{tikzpicture}
		\caption{The cut from $\alpha_2$ to $\alpha_3$ adds the gray area to the competitor producing a better Cheeger ratio.}\label{fig:shortcut}
	\end{figure}	
By Step one, the set $\pa E\cap \Om\cap U$ is given by the union of at most two families $\{\alpha_{j}\}_{j}$ and $\{\beta_{j}\}_{j}$ of arcs, with endpoints equal to, respectively, $a_{j,1}=\sigma(t_{j,1})$, $a_{j,2}=\sigma(t_{j,2})$, $b_{j,1}=\sigma(s_{j,1})$, $b_{j,2}=\sigma(s_{j,2})$. Of course, when $x$ is the endpoint of exactly one arc (respectively, two arcs) then at least one of these two (respectively, both) families have a single element, but the subsequent argument is not affected by this possibility. We thus assume without loss of generality that both families contain countably many elements. By properties (ii) and (iii) proved in Step one, and setting $A_{j} = (t_{j,1},t_{j,2})$ and $B_{j} = (s_{j,1},s_{j,2})$, we deduce that 
	\[
	t_{j,i}\le 0\le s_{j,i}\qquad\text{and}\qquad A_{j}\cap A_{k} = B_{j}\cap B_{k} = \emptyset\,,
	\]
	for all $j\neq k\in \N$ and $i=1,2$. It is then clear how to define a parametrization of $\pa E\cap U$ by concatenating circular arcs with pieces of $\sigma$. To do this, we may assume that every arc $\alpha_{j}$ (respectively, $\beta_{j}$) is parametrized over the interval $A_{j}$ (respectively, $B_{j}$), so that we can define for $t\in (-1,1)$
	\[
	\g(t) = \begin{cases}
	\alpha_{j}(t) & \text{if }t\in A_{j},\\
	\beta_{j}(t) & \text{if }t\in B_{j},\\
	\sigma(t) & \text{if }t\notin \bigcup_{j\in \N} (A_{j}\cup B_{j}).
	\end{cases}
	\]
	It is easy to check that $\g$ is well-defined in $(-1,1)$ and provides a continuous, injective parametrization of $\pa E\cap U$. This completes the proof of the proposition.
\end{proof}

\begin{proposition} \label{prop:sconnectednessE}
Suppose $\Om \subset \R^2$ is open, bounded, and simply connected with $|\pa \Om| = 0$. Then any Cheeger set $E$ of $\Om$ is Lebesgue-equivalent to a finite union of simply connected open sets.
\end{proposition}

\begin{proof} 

\textit{Step one.} We show that $E$ is Lebesgue-equivalent to the set $E^{\circ}$ of its interior points. It is enough to check that, given a sequence $\{\Om_{j}\}_{j\in \N}$ of open subsets that are relatively compact in $\Om$ with $\Om = \bigcup_{j\in \N}\Om_{j}$, one has 
\begin{equation}\label{eq:Einterno}
|(E\setminus E^{\circ})\cap \Om_{j}| = 0\qquad \text{for all }j\,.
\end{equation}
Indeed, if we combine \eqref{eq:Einterno} with the assumption $|\pa \Om|=0$ we get
\[
|E\setminus E^{\circ}|  \le |\pa E\cap \pa \Om| + \sum_{j\in \N} |(E\setminus E^{\circ})\cap \Om_{j}| = 0\,.
\]
On the other hand, \eqref{eq:Einterno} holds since $\pa E\cap \Om_{j}$ coincides with the intersection of a finite union of circular arcs with $\Om_{j}$, thus in particular $|\pa E \cap \Om_{j}|=0$ for all $j$.

\textit{Step two.} Let $G= \Om \setminus E$. We show the following fact: if $G'$ is a P-indecomposable component of $G$ then 
\[
P(G';\Om) < P(G')\,.
\] 
Otherwise, if $P(G';\Om) = P(G')$, we would have 
\begin{align*}
P(E\cup G') &= P(E\cup G';\pa \Om) + P(E\cup G';\Om)\\
&\le P(E;\pa \Om) + P(E;\Om) + P(G';\Om) - 2\Hau^{1}(\pa^{*}E\cap \pa^{*}G'\cap \Om) \\
&= P(E) - \Hau^{1}(\pa^{*}E\cap \pa^{*}G'\cap \Om)\\ 
& \le P(E)\,,
\end{align*}
which would in turn contradict the fact that $E$ is a Cheeger set, as $|E\cup G'|>|E|$. 

\textit{Step three.} Assuming $E$ open according to Step one, we let $\g$ be a nontrivial simple closed curve contained in $E$ and let $F$ be the bounded subset of $\R^2$ with $\pa F = \g$. Note that $F$ is compactly contained in $\Om$ since its closure is a compact set with boundary $\pa F = \g$ at a positive distance from $\pa E$. We want to prove that $F$ is compactly contained in $E$ as well, which amounts to showing that $|F\cap G| = 0$. Argue by contradiction and suppose $|F\cap G|>0$. Since $\gamma$ is at a positive distance from $\pa E$, $F\cap G \supseteq G'$ for some P-indecomposable component $G'$ of $G$. Then, as $G' \subset F \subset\subset \Om$, a contradiction is reached by Step two.

Hence, by Steps one and three and up to a set of measure zero, $E$ is the countable union of open and simply connected sets. Finally, each connected component of $E$ is itself a Cheeger set, thus has volume bounded from below. Pairing this with the boundedness of $\Om$, we find that $E$ has finitely many connected components, completing the proof.
\end{proof}

\subsection{Sets with positive reach and Steiner's formulas}

Given $A\subset \R^n$, the parallel set of $A$ at distance $r$ is defined by
\[
A_r  = \{ x \in \R^n : \dd(x, A ) \leq r\}.
\]  
For a convex set $A\subset \R^n$, the Hausdorff $d$-dimensional measure of $A_r$ can be expressed as a degree-$d$ polynomial in $r$ with coefficients depending on $A$ for any $r>0$. This was originally shown in \cite{Steiner1840},  cf. \cite{Schneider14}. For $n=d=2$, this polynomial takes the form
\begin{equation}\label{eqn: Steiner}
|A_r | = |A| + rP(A) + \pi r^2.
\end{equation}
Polynomial expansions of the same type were shown for $C^2$ sets in \cite{Weyl39} for $r>0$ sufficiently small. If $A\subset \R^2$ is simply connected and of class $C^2$ (actually, $C^{1,1}$ will suffice), then the expansion holds in the same form \eqref{eqn: Steiner}.
In \cite{Federer59}, Federer gave a unified treatment of this theory  with the introduction of sets of positive reach. He defined the reach of a set $A$ to be
\[
\text{reach}(A) = \sup \{ r : \text{ if } x \in A_r, \text{ then } x \text{ has a unique projection onto } A\}
\]
and showed a polynomial expansion for $|A_r|$ for $0< r <\text{Reach}(A)$. 

If $A\subset \R^2$ is simply connected with positive reach, the proof of this polynomial expansion is fairly simple.
For any $0<t<\text{reach}(A)$, $A_t$ is a simply connected set of class $C^{1,1}$. Hence, \eqref{eqn: Steiner} holds for $A_t$, that is, for $t<r$,
\begin{equation}\label{eqn: 1}
|A_r|-|A_t| = (r-t) P(A_t) + \pi (r-t)^2.
\end{equation}
Since $\lim_{t\to0^+}|A_t|= |A|$, it follows that  $ \lim_{t\to 0^+}P(A_t)=c_A$ exists. Hence, taking $t \to 0^+$ in \eqref{eqn: 1}, we see that \eqref{eqn: Steiner} holds for $A$ with $c_A$ replacing $P(A)$.
And actually, dividing \eqref{eqn: 1} by $r-t$, then letting $t\to 0^+$ and $r\to 0^+$ respectively, it follows that  $ c_A= \M_o(A)$, where the outer Minkowski content\footnote{Under certain regularity assumptions on $A$, the outer Minkowski content of $A$ is equal to the perimeter of $A$; see \cite{AmbColVil} for a treatment of the subject. In certain cases arising in our setting, these two quantities fail to coincide, but this is of no importance in our application.} $\M_o(A)$ of $A$ is defined by
\[
\M_o(A) =  \lim_{r\to 0^+} \frac{|A_r| - |A|}{r}.
\]
So, if $A\subset \R^2$ is a simply connected set with positive reach, then
\begin{equation}\label{eqn: Steiner2}
|A_r| = |A|+  \M_o(A)r  + \pi r^2 \qquad  0<r <\text{reach}(A).
\end{equation}
Differentiating this identity, we also find that
\begin{equation}\label{eqn: Steiner3}
P(A_r) = \M_o(A) + 2 \pi r \qquad 0<r <\text{reach}(A).
\end{equation}

\section{Proof of Theorem~\ref{theorem:Pestov and Ionin}} 
\label{section:3}
The goal of this section is to prove Theorem \ref{theorem:Pestov and Ionin}. Throughout the section, we will make the following assumption:
\begin{equation}\label{eqn:JV}
\begin{split}
&\text{\it $E\subset \R^2$ is a Jordan domain with curvature bounded}\\
&\text{\it from above by $1$ in the viscosity sense.}
\end{split}
\end{equation}
By Proposition \ref{prop:Jordan}, we know that $\pa E$ is locally homeomorphic to an interval. We thus set $r_{0}$ as in Definition \ref{def:lochomeo} and, without loss of generality, we assume $\omega_{0}(r_{0})<1$ and $r_{0}<1/4$. 
In order to prove Theorem~\ref{theorem:Pestov and Ionin}, we must show that $E$ contains at least one point $x$ such that $\dd(x,\pa E)\geq 1$.
For any $x \in E$, we define the projection set
\begin{equation}\label{eq:set of proj}
P_x = \{ y \in \pa E : \dd (x, \pa E) = |x-y|\}\,.
\end{equation}

The following lemma will play a key role in what follows.
\begin{lemma}\label{lem: dist 1}Suppose $E \subset \R^2$ satisfies \eqref{eqn:JV}. Suppose there exist $x\in E$ and $y_1, y_2 \in P_x$ such that $0<|y_{1}-y_{2}|<r_{0}$. Then $\dd(x, \pa E) \geq 1$.
\end{lemma} 

\begin{proof} 
Let $\g\subset \pa E$ be the curve homeomorphic to a closed interval and connecting $y_{1}$ and $y_{2}$, according to Definition \ref{def:lochomeo}. We assume for the sake of contradiction that $\rho = \dd(x,\pa E) <1$. Fix $1>r>\max(\rho,\omega_{0}(r_{0}))$, set $m = (y_{1}+y_{2})/2$, and denote by $\ell$ the line orthogonal to $y_{1}-y_{2}$ passing through $m$. 
Choose a point $z\in \ell$ in such a way such that $|z-y_{1}| = r$ and,  for some $t>0$,
\begin{align}\label{eq:q}
q=m+t(m-z)\in \g\cap \ell\,.
\end{align}
We remark that there are either one or two possible choices of $z$.
 Let $A$ be the uniquely defined half of $\pa B_{r}(z)$ which contains $y_{i}$ for $i=1,2$ and is symmetric with respect to $\ell$. We consider the translated arc $A_{s} = A +sv$ with $s\ge 0$ and $v = m-z$. Observe that $A_{s} \cap \g$ is empty when $s$ is sufficiently large. Moreover, the maximum distance of $\g$ from the line $\ell$ is bounded from above by $\omega_{0}(r_{0})$, hence the distance of each endpoint of $A_{s}$ from $\g$ is bounded from below by $r-\omega_{0}(r_{0})>0$ for all $s\ge 0$. 
Then, we let
\[ 
\tilde s = \inf\{ s:  s>0,\ A_{s} \cap \g = \emptyset\}.
\]
The set $F =\g \cap A_{\tilde s}$ is nonempty and consists of points in the relative interior of $A_{\tilde s}$. If $\tilde s = 0,$ then $y_1, y_2 \in F$. Let $p$ be the unique point of $A\cap \ell$ and note that,  since $r>\rho$, we have $p\in B_{\rho}(x)$. Then the point $q$  defined in \eqref{eq:q}, which lies on the segment connecting $m$ and $p$, is contained in $B_{\rho}(x)$. We thus reach a contradiction to the fact that $\g \cap B_{\rho}(x)$ is empty. So, $\tilde s>0$. Since $F$ now only contains points in the relative interior of $\g$, there exists $w\in F$ such that $A_{\tilde s}$ touches $E$ locally from outside at $w$. This contradicts the bound on the curvature of $E$ in the viscosity sense. 
 \end{proof}

We introduce the following definitions:
\begin{definition}\label{def:cutfocal}
Let $x\in E$. Given $y \in P_x$, we let $z_{x,y}(t) := x+t(x-y).$
\begin{enumerate}
\item We call $x$ a cut point if there exists $y \in P_x$ such that 
\[
 \sup \left\{ t:  |z_{x,y}(t) -y | = \dd (z_{x,y}(t) , \pa E)\right\} = 0.
\] 
\item We call $x$ a focal point if there exists $y \in P_x$ such that
\[
\sup \left\{ t: y \text{ is a local minimizer of }\dd(z_{x,y}(t) , \, \cdot \,) \text{ among points in } \pa E\right\} =0.
\]
We call such a $y \in P_x$ a {\it focal projection} of $x$.
\end{enumerate}
\end{definition}
We let $\mathcal{C}$ denote the set of cut points of $E$. Note that any focal point is also a cut point. Furthermore, if $\# P_x >1,$ then $x$ is a cut point, and if $\# P_x = \infty$, then $x$ is a focal point. An example of the cut locus for a set is given in Figure~\ref{fig:cutlocus}.

\begin{figure}[t]
\centering
\includegraphics[scale=1]{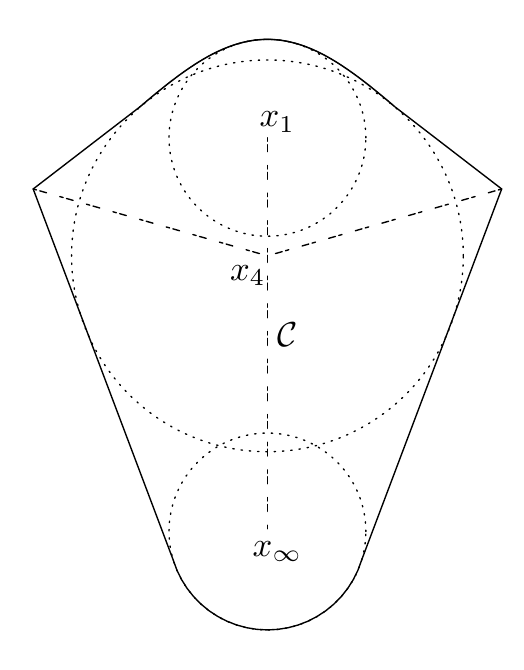}
\caption{The $4$ dashed lines form the cut locus $\mathcal{C}$ of the depicted set. The point $x_1$ is a focal point with unique projection, the point $x_\infty$ is a focal point with infinite projections and the point $x_4$ is a cut point with $4$ projections. All the other points in $\mathcal{C}$ have $2$ projections.} \label{fig:cutlocus}
\end{figure}

Given $\e>0$ and $y \in \pa E$, we denote by $\Sigma_{\e}(y)$ a portion of $\pa E$ containing $y$ with diameter less than $\e$ that is homeomorphic to an open interval. 

\begin{remark}[Geometric interpretation of focal points]\label{rmk:geomfocal} {\rm
Let $x$ be a focal point and let $y\in P_x$ be a focal projection. Let $d = |x- y|$ and $d_t = |z_{x,y}(t) -y|$, so that $ y \in \pa B_d(x) \cap \pa B_{d_t}(z_{x,y}(t))\cap \pa E$. By the definition of focal point, we have
\[
\Sigma_\e (y)
  \cap B_{d}(x) = \emptyset, \qquad 
\Sigma_\e(y)
 \cap B_{d_t}(z_t) \neq \emptyset  \qquad \text{for all } t>0.
\]
}
\end{remark}

\begin{remark}[Nesting property] \label{rmk: nest}{\rm
With the geometric interpretation of focal point in mind, we observe the following nesting property. If there exists $\tau>0$ such that $ y \in \pa E$ minimizes $\dd(z_{x,y}(\tau), \, \cdot \,)$ in $ \Sigma_\e( y)$, then $y$ also minimizes $\dd(z_{x, y}(t), \, \cdot \,)$ in $\Sigma_\e(y)$ for all $ t <\tau.$ Indeed, $\Sigma_\e( y)$ is disjoint from $B_{d}(z_{x, y}(\tau))$ with $d=|z_{x, y}(\tau)-y|$, so it is also disjoint from the ball of smaller radius $|z_{x,y}(t) - y|$ centered at $z_{x, y}(t)$, which is obviously contained in $B_{d}(z_{x, y}(\tau))$.
}
\end{remark}

We have the following corollary of Lemma~\ref{lem: dist 1}.
\begin{corollary}\label{cor: projections} Assume $E \subset \R^2$ satisfies \eqref{eqn:JV}, and suppose $x \in E$ is a cut point with $\dd(x, \pa E) < 1$ and $y \in P_x$. For any $0 < \e < r_0$, with $r_0$ as in Definition~\ref{def:lochomeo}, there exists $\delta_{\e}>0$ such that  every $z \in B_{\delta_{\e}}(x)$ has a unique projection onto $\Sigma_\e(y)$. 
\end{corollary}

\begin{proof} 
Fix $\delta_{\e}>0$ small enough so that $d_z= \dd(z,\Sigma_\e(y))<1$ for all $z \in B_{\delta_{\e}}(x).$ Assume that $y_1,y_2 \in \Sigma_\e(y) \cap \pa B_{d_z}(z)$ for some $z\in B_{\delta_{\e}}(x)$. This implies that $|y_1-y_2|\le \e<r_0$. Hence,  $y_1=y_2$ by Lemma~\ref{lem: dist 1}.
\end{proof}

Next, we show that focal points lie at distance at least $1$ from $\pa E$.

\begin{lemma}\label{prop: focal point distance} Suppose $E \subset \R^2$ satisfies \eqref{eqn:JV} and let $r_{0}$ be as in Definition \ref{def:lochomeo}. If $x \in E$ is a focal point, then $\dd(x, \pa E) \geq 1.$
\end{lemma}

\begin{proof}
Let $y \in P_x$ be a focal projection for $x$. Let $d = |x- y|$ and $d_t = |z_{x,y}(t) -  y|.$
We show that $d_t  \geq 1$ for all $t>0$.

First, suppose that $\Sigma_\e( y)$ is contained in $\overline{ B}_{d_t}(z_{x,y}(t))$ for $\e>0$ sufficiently small. In this case, $\pa B_{d_t}(z_{x,y}(t))$ in a neighborhood of $y$ is a circular arc touching $E$ from outside at $ y$. So, $d_t \geq 1$ as the curvature of $E$ is bounded above by $1$ in the viscosity sense.

Next, suppose that $\Sigma_\e(y) \setminus \overline{ B}_{d_t}(z_{x,y}(t))$ is nonempty for every $\e>0$. Then, since $x$ is a focal point, for $\e<r_0$, there exist distinct points $y_1,y_2 \in \Sigma_\e(y) \cap \pa  B_{d_t}(z_{x,y}(t))$ such that the following property holds: if we let $\g$ denote the subset of $\Sigma_\e(y)$ that is homeomorphic to an interval and has endpoints $y_1$ and $y_2$, then
$\g \cap  B_{d_t}(z_{x,y}(t))$ is empty. 
 Now, arguing 
by contradiction exactly as in the proof of Lemma ~\ref{lem: dist 1} and taking $r = \max\{ d_t, r_0\}$, we finally show that $r\geq 1$. As $r_0 <1/4$, we determine that $d_t \geq 1$, concluding the proof.
\end{proof}

\begin{lemma}\label{prop: at least one} Suppose $E \subset \R^2$ satisfies \eqref{eqn:JV} and let $x \in E $ be a cut point. Then at least one of the following alternatives holds:
\begin{enumerate}
\item $\# P_x >1$,
\item $x$ is a focal point.
\end{enumerate}
\end{lemma}
\begin{proof} 
We show that if $x $ is a cut point with $\# P_x =1$, then $x$ is necessarily a focal point. Let $y$ be the unique point in $P_x$. We first prove the following claim: {\it for all $\e>0$, there exists $t_\e>0$ such that}
\[
P_{z_{x,y}(t)} \subset \Sigma_\e (y) \qquad \forall\, 0 <t < t_\e\,.
\] 
If not, we may find $\e>0$, $\{t_n\},$ and $\{y_n\}$ such that 
\[
0<t_n \to 0, \qquad y_n \in P_{z_{x,y}(t_n)}, \qquad|y_n -y|\ge \e\,.
\]
Since $\pa E$ is compact, up to a not-relabeled subsequence, $y_n \to y_1 \in \pa E$ with $y_1 \not\in \Sigma_{\e}(y)$. In particular, $y_1 \neq  y$. On the other hand,
we deduce that $y_1 \in P_x$, because, by continuity, we have
\[
\dd(z_{x,y}(t_n) , \pa E) \to \dd(x, \pa E) = |x- y|
\]
and 
\[
\dd(z_{x,y}(t_n) , \pa E)  = |z_{x,y}(t_n) -y_n|  \to |x- y_1|\,.
\]
Thus $|x-y_{1}| = |x-y|$, which contradicts $\#P_x =1$ and demonstrates our claim.

Now we let $d_t =| z_{x,y}(t) - y|$ and assume for contradiction that $x$ is not a focal point. Then, for $\e>0$ and $t>0$ small, the intersection of $\Sigma_\e(y)$ and $B_{d_t}(z_{x,y}(t))$ must be empty. 
Then, because $x$ is a cut point, $ y \not\in P_{z_{x,y}(t)}$, and therefore $ P_{z_{x,y}(t)}\cap \Sigma_\e(y)$ is empty, for any $t >0$. We have thus reached a contradiction with the previous claim, and the proof is complete.
\end{proof}

\begin{proposition}\label{lem: structure}Suppose $E \subset \R^2$ satisfies \eqref{eqn:JV}. Suppose $x \in E$ is a cut point with $\dd(x, \pa E) <1$ and $\# P_x =k \geq 2.$ Then, there exists $\delta>0$ such that $\mathcal{C} \cap B_{\delta}(x)$ is the union of $k$ curves of class $C^{1}$ meeting at $x$. \end{proposition}
\begin{proof}
Let $P_x = \{y_1, \dots, y_k\}$ and $d = \dd(x, \pa E)$ and, for any $\e>0$ and $i=1,\dots,k$, define $\Sigma_{i} = \Sigma_{\e}(y_{i})$. Choose $\e < r_0/4$ small enough such that $\Sigma_i$ and $\Sigma_j$ have a positive distance for all $i \neq j$.  

Observe that there exists $r>d$ such that $B_r(x) \cap \pa E \subset  \bigcup_{i=1}^k \Sigma_i$. Hence, for $\delta>0$ sufficiently small, 
\[
P_z \subset \bigcup_{i=1}^k \Sigma_i
\]
 for all $z \in B_{\delta}(x)$. 
Up to further decreasing $\delta$, we may also assume that  $0<\delta <\delta_\e$, where $\delta_\e$ is as in Corollary~\ref{cor: projections}. Consequently, each $z\in B_\delta(x)$ has a unique projection $y_z^i$ onto $\Sigma_i$ for each $i \in \{1, \dots, k\}$. The functions $\rho_i: B_\delta(x) \to \R$ defined by
\[
\rho_i(z) = \dd(z, \Sigma_i)
\]
are therefore continuously differentiable at every $z \in B_\delta(x)$ with
\[
\na \rho_i (z) = \frac{z-y_z^i}{|z-y_z^i|} \qquad \forall z \in B_\delta(x)
\]
We remark that the continuity of the differential depends on the continuity of the map $z\mapsto y^{i}_{z}$, which in turn follows from the uniqueness of the projection of $z$ onto $\Sigma_i$.

Now we split the proof into two cases.

\medskip
\noindent{\it Case 1: $k  =2$.} Consider the differentiable function $f: B_\delta(x) \to \R$ defined by $f(z) = \rho_1(z) - \rho_2(z)$. 
Note that $\nabla f(x) \neq 0$, hence, up to a rotation and reducing $\delta$, we can assume that $\pa_{z_1} f(z)>0$ for all $z \in B_{\delta}(x).$ Applying the implicit function theorem, we see that the set $\{ f=0\}$ is the graph of a $C^{1}$ function of $z_2$ defined in a neighborhood of $x_{2}$. Since $\{f = 0\} = \mathcal{C} \cap B_\delta(x)$, the proof is complete in this case. 

\medskip
\noindent{\it Case 2: $k  >2$.} The union of the segments $[y_i,x]$ divide $B_d(x)$ into $k$ circular sectors $S_{1},\dots,S_{k}$. We possibly relabel the points $y_{i}$ in such a way that $y_{i}$ and $y_{i+1}$ are associated with the sector $S_{i}$ for $i=1,\dots k$ (where $k+1$ is identified with $1$). 
We claim that if $z \in B_{\delta}(x) \cap S_{i}\setminus \{x\},$ then $P_z \subset \Sigma_i \cup \Sigma_{i+1}$. Indeed, we already know that $P_{z}\subset \bigcup_{j=1}^{k}\Sigma_{j}$, thus we only have to show that whenever $j\neq i,i+1$ we have $P_{z}\cap \Sigma_{j} = \emptyset$. By contradiction, suppose that there exists $y\in P_{z}\cap \Sigma_{j}$ for some $j\neq i,i+1$. In this case the segment $[y,z]$ intersects either $[y_{i},x]$ or $[y_{i+1},x]$ at some $\bar z$. Without loss of generality, let us suppose that it is $[y_{i},x]$. First we notice that $\bar z\neq x$, otherwise we would have that $P_{x} = \{y\}$, i.e., a contradiction with the assumption on $x$. On the other hand, for $\bar z\neq x$ the ball $B_{|\bar z - y_{i}|}(\bar z)$ is contained in $B_{d}(x)$ and the boundaries of these balls touch only at $y_{i}$. Hence,
we obtain
\[
|z - y_{i}|\le |z-\bar z| + |\bar z-y_{i}| < |z-\bar z| + |\bar z-y| = |z-y|\,,
\]
contradicting the fact that $y\in P_{z}$. This shows our claim. 
Arguing as in Case 1 we find that the set of points that are equidistant from $\Sigma_i$ and $\Sigma_{i+1}$ form a $C^{1}$ curve $\g_{i}$ in $B_{\delta}(x)$, and the previous claim implies 
\begin{equation}\label{eq:g12inC}
\gamma_{i}\cap B_{\delta}(x) \cap S_{i}\ \subset\ \mathcal{C}\cap B_{\delta}(x) \cap S_{i}\,.
\end{equation}
We now show the opposite inclusion in \eqref{eq:g12inC}. If $z\in \mathcal{C}\cap B_{\delta}(x) \cap S_{i}$, since $\dd(z,\pa E) <1$, we deduce from Lemma \ref{prop: focal point distance} that $z$ cannot be a focal point. Thus, owing to Lemma~\ref{prop: at least one} we infer that $\# P_{z} \ge 2$. As previously noted, $z$ must have a unique projection onto $\Sigma_{i}$ and onto $\Sigma_{i+1}$, so $z$ must have exactly one projection in each. This means that $z\in \g_{i}$, as desired, so
\[
\gamma_{i}\cap B_{\delta}(x) \cap S_{i}\ =\ \mathcal{C}\cap B_{\delta}(x) \cap S_{i}\,.
\]
Repeating the previous argument for each sector $S_{i}$ we finally obtain that $\mathcal{C}\cap B_{\delta}(x)$ is the union of exactly $k$ curves $\gamma_{1},\dots,\gamma_{k}$ of class $C^{1}$ meeting at $x$. 
\end{proof}

We are now ready to prove Theorem~\ref{theorem:Pestov and Ionin}.

\begin{proof}[Proof of Theorem~\ref{theorem:Pestov and Ionin}] Up to rescaling, we may assume that $h=1$. Suppose for the sake of contradiction that $\dd(x, \pa E) <1$ for all $x \in \overline E$. In particular, by Lemma~\ref{prop: focal point distance}, $\C$ contains no focal points, and so owing to Lemma~\ref{prop: at least one}, $1< \# P_x <\infty$ for all $x \in \C.$ Then, Proposition~\ref{lem: structure} implies that $\C$ is the union of $C^{1}$ curves. In particular, $\C$ is a topological graph where each vertex has finite valence.

We claim that $\C$ contains no non-trivial loops.
Indeed, $E$ is simply connected, so any $x$ in the interior of a loop is contained in $E$. Take $y \in P_x$, let $\ell$ be the line segment between $x$ and $y$, and note that there is some $z\in \ell \cap \C$. Since $z \in \C$, there exists some $y_z \neq y$ in $P_z$. So,
  $$|x- y_z | \leq |x-z| + |z- y_z | \leq |x-z| + |z-y| = |x-y| = \dd(x, \pa E).$$
 It follows that $y_z \in P_x$, so $x \in \C$.
Therefore, $\C$ is a union of disjoint trees. 

We now claim that $\mathcal{C}$ is compact. Note that $\mathcal C$ is bounded and consider a sequence of cut points $\{x_i\}$ converging to some $x$, such that $\dd(x_i, \pa E)<1$ and $x_{i}$ is not a focal point for any $i$. As $\{x_i\}\subset \mathcal{C}$, we can take two sequences $\{y_i\}$ and $\{w_i\}$ such that $y_i \neq w_i$ and $|x_i - y_i|=|x_i-w_i| = \dd(x_i, \pa E)$ for all $i$. By compactness, $y_i \to y \in \pa E$ and $w_i \to w \in \pa E$ up to subsequences. Furthermore, $y\neq w$; otherwise we would have $|y_i - w_i| \to 0$, and so Lemma \ref{lem: dist 1} would imply $\dd(x_{i}, \pa E) \geq 1$ for $i$ large enough, against our assumption. By the continuity of the distance function, $\{y,w\} \subset P_x$. So, $x$ is a cut point and $\C$ is compact. Then, $\C$ is necessarily a \emph{finite} union of trees, thanks to the local structure of the cut locus as given by Proposition \ref{lem: structure}. Finally, $\C$ must have at least one vertex $v$ of valence $1$, and again by Proposition~\ref{lem: structure}, $\dd(v, \pa \Om) \geq 1$, yielding a contradiction.
\end{proof}

\section{Proof of Theorem~\ref{lem:eraser}}\label{sec: eraser}

We will now prove Theorem \ref{lem:eraser}, which seems interesting in its own right and may be a useful tool in other settings. In our setting, it allows us to drop the regularity of the curve of centers that is considered in the rolling ball lemma, Lemma~\ref{lem:rollingball}. We recall that it states that, given a planar curve $\g$ and a constant $r>0$, one can regularize $\g$ to a curve $\sigma$ of class $C^{1,1}$ with curvature bounded by $1/r$ and same endpoints as $\g$ such that the set swept by a ball of radius $r$ rolled along $\g$ contains the set swept by a ball with the same radius rolled along $\sigma$, as shown in Figure~\ref{fig:eraser}.
\begin{figure}[t]
\centering
\definecolor{ffffff}{rgb}{1.,1.,1.}
\definecolor{grigio}{rgb}{.9,.9,.9}
\begin{tikzpicture}[line cap=round,line join=round,>=triangle 45,x=.5cm,y=.5cm]
\clip(-5.2,-3.2) rectangle (7.2,6.2);

\fill[fill=grigio] (-4.01,3.) -- (-1.99,3.) -- (-1.99,1.) -- (-4.01,1.) -- cycle;
\fill[fill=grigio] (0.,5.01) -- (0.,2.99) -- (-2.,2.99) -- (-2.,5.01) -- cycle;
\fill [shift={(-1.,5.)},fill=grigio]  (0,0) --  plot[domain=0.:3.141592653589793,variable=\t]({1.*1.*cos(\t r)+0.*1.*sin(\t r)},{0.*1.*cos(\t r)+1.*1.*sin(\t r)}) -- cycle ;
\fill [shift={(-4.,2.)},fill=grigio]  (0,0) --  plot[domain=1.5707963267948966:4.71238898038469,variable=\t]({1.*1.*cos(\t r)+0.*1.*sin(\t r)},{0.*1.*cos(\t r)+1.*1.*sin(\t r)}) -- cycle ;

\fill[color=black,fill=black,fill opacity=0.4] (-1.,-1.) -- (0.,-2.) -- (-1.,-2.) -- cycle;
\fill[color=black,fill=black,fill opacity=0.4] (0.,2.) -- (-1.,3.) -- (-2.,2.) -- (-1.,1.) -- cycle;

\draw [line width=0.4pt,dash pattern=on 2pt off 2pt] (-2.,2.)-- (4.,2.);
\draw [shift={(4.,0.)},line width=0.4pt,dash pattern=on 2pt off 2pt]  plot[domain=-1.5707963267948966:1.5707963267948966,variable=\t]({1.*2.*cos(\t r)+0.*2.*sin(\t r)},{0.*2.*cos(\t r)+1.*2.*sin(\t r)});
\draw [line width=0.4pt,dash pattern=on 2pt off 2pt] (4.,-2.)-- (-1.,-2.);
\draw [line width=0.4pt,dash pattern=on 2pt off 2pt] (-1.,-2.)-- (-1.,3.);

\draw [shift={(4.,0.)}] plot[domain=-1.5707963267948966:1.5707963267948966,variable=\t]({1.*3.*cos(\t r)+0.*3.*sin(\t r)},{0.*3.*cos(\t r)+1.*3.*sin(\t r)});
\draw [shift={(4.,0.)}] plot[domain=-1.5707963267948966:1.5707963267948966,variable=\t]({1.*1.*cos(\t r)+0.*1.*sin(\t r)},{0.*1.*cos(\t r)+1.*1.*sin(\t r)});
\draw (4.,-1.)-- (0.,-1.);
\draw (4.,1.)-- (0.,1.);
\draw (4.,3.)-- (0.,3.);
\draw (0.,1.)-- (0.,-1.);
\draw (4.,-3.)-- (-1.,-3.);
\draw (-2.,5.)-- (-2.,3.);
\draw (-2.,3.)-- (-4.,3.);
\draw (-4.,1.)-- (-2.,1.);
\draw (-2.,1.)-- (-2.,-2.);
\draw (0.,5.)-- (0.,3.);
\draw [shift={(-1.,5.)}] plot[domain=0.:3.141592653589793,variable=\t]({1.*1.*cos(\t r)+0.*1.*sin(\t r)},{0.*1.*cos(\t r)+1.*1.*sin(\t r)});
\draw [shift={(-4.,2.)}] plot[domain=1.5707963267948966:4.71238898038469,variable=\t]({1.*1.*cos(\t r)+0.*1.*sin(\t r)},{0.*1.*cos(\t r)+1.*1.*sin(\t r)});
\draw [shift={(-1.,-2.)}] plot[domain=3.141592653589793:4.71238898038469,variable=\t]({1.*1.*cos(\t r)+0.*1.*sin(\t r)},{0.*1.*cos(\t r)+1.*1.*sin(\t r)});

\draw [shift={(0.,-1.)},draw=none,fill=ffffff,fill opacity=1.0]  (0,0) --  plot[domain=3.141592653589793:4.71238898038469,variable=\t]({1.*1.*cos(\t r)+0.*1.*sin(\t r)},{0.*1.*cos(\t r)+1.*1.*sin(\t r)}) -- cycle ;

\draw [shift={(0.,3.)},draw=none,fill=ffffff,fill opacity=1.0]  (0,0) --  plot[domain=3.141592653589793:4.71238898038469,variable=\t]({1.*1.*cos(\t r)+0.*1.*sin(\t r)},{0.*1.*cos(\t r)+1.*1.*sin(\t r)}) -- cycle ;
\draw [shift={(0.,1.)},draw=none,fill=ffffff,fill opacity=1.0]  (0,0) --  plot[domain=3.141592653589793:4.71238898038469,variable=\t]({1.*1.*cos(\t r)+0.*1.*sin(\t r)},{0.*1.*cos(\t r)+-1.*1.*sin(\t r)}) -- cycle ;
\draw [shift={(-2.,1.)},draw=none,fill=ffffff,fill opacity=1.0]  (0,0) --  plot[domain=3.141592653589793:4.71238898038469,variable=\t]({-1.*1.*cos(\t r)+0.*1.*sin(\t r)},{0.*1.*cos(\t r)+-1.*1.*sin(\t r)}) -- cycle ;
\draw [shift={(-2.,3.)},draw=none,fill=ffffff,fill opacity=1.0]  (0,0) --  plot[domain=3.141592653589793:4.71238898038469,variable=\t]({-1.*1.*cos(\t r)+0.*1.*sin(\t r)},{0.*1.*cos(\t r)+1.*1.*sin(\t r)}) -- cycle ;
\draw [shift={(-2.,3.)},draw=none,fill=black,fill opacity=0.1]  (0,0) --  plot[domain=-1.5707963267948966:0.,variable=\t]({1.*2.*cos(\t r)+0.*2.*sin(\t r)},{0.*2.*cos(\t r)+1.*2.*sin(\t r)}) -- cycle ;

\draw [color=black] (-4.,2.)-- (-2.,2.);
\draw [color=black] (-1.,3.)-- (-1.,5.);
\draw [shift={(-2.,3.)},color=black]  plot[domain=-1.5707963267948966:0.,variable=\t]({1.*1.*cos(\t r)+0.*1.*sin(\t r)},{0.*1.*cos(\t r)+1.*1.*sin(\t r)});
\draw (4.,1.8) node[anchor=north west] {$\gamma$};
\draw (-2,3.4) node[anchor=north west] {$\sigma$};
\end{tikzpicture}
\caption{The dashed curve $\g$ and the curve $\sigma$ of Theorem~\ref{lem:eraser}. The light-grayed-out area is the region swept by rolling the ball along $\sigma$. The dark-grayed-out areas are the connected components of the interior of the inner retract $Y$ in the proof of the theorem.} \label{fig:eraser}
\end{figure}
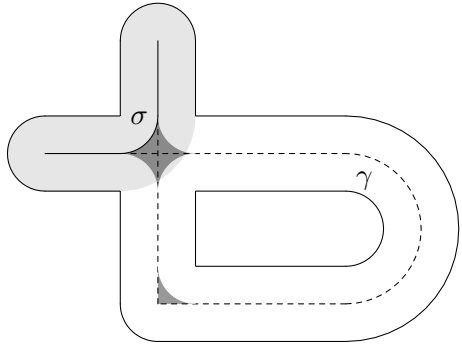

\begin{proof}[Proof of Theorem~\ref{lem:eraser}]
We may without loss of generality assume that $\g(0)\neq \g(1)$ 
and that $r=1$. We then split the proof in two steps.

\textit{Step one.} We make the additional assumption that $\g$ has finite length. Let $Y_{\g,1}$ denote the closed set of points $x\in U_{\g,1}$ such that $\dd(x,\pa U_{\g,1}) \ge 1$. Of course, $\g(t)\in Y_{\g,1}$ for all $t\in [0,1]$. Consider a continuous curve $\sigma:[0,1]\to \R^{2}$ that minimizes the length among curves with $\sigma(t)\in Y_{\g,1}$ for all $t\in [0,1]$ satisfying $\sigma(0)=\g(0)$ and $\sigma(1)=\g(1)$. As the set $Y_{\g,1}$ is compact, the existence of such a minimizer $\sigma$ is standard, and we may assume $|\sigma'(t)| = C$ for some constant $C>0$ and for almost all $t$. Observe that the inclusion $\sigma(t) \in Y_{\g,1}$ for all $t\in [0,1]$ implies the required property $U_{\sigma,1} \subset U_{\gamma,1}$. We are thus left to prove that $\sigma$ is $C^{1,1}$ with the modulus of its curvature bounded by $1$. From now on, we shall simply write $U = U_{\sigma,1}$ and $Y = \{x\in U:\ \dd(x,\pa U)\ge 1\}$.

Before going on, we introduce two definitions. We say that $\sigma$ has bilateral (tangent) balls of radius $1$ at $x=\sigma(t)$ if there exist $p,q\in \R^{2}$ such that 
\begin{equation}\label{eqn: tangent balls}
\{x\} = \pa B_{1}(p)\cap \pa B_{1}(q),\qquad B_{1}(p)\cap B_{1}(q) = \emptyset, 
\end{equation}
and $\sigma \cap \big(B_{1}(p)\cup B_{1}(q)\big) = \emptyset.$
 When no confusion arises, we will simply say that \textit{$\sigma$ satisfies the bilateral ball condition at $x$}. Then $\sigma$ will be said to satisfy a \textit{local} bilateral ball condition at $x=\sigma(t)$ if there exist $p,q\in \R^{2}$ such that \eqref{eqn: tangent balls} holds and 
 $\sigma(s) \notin B_{1}(p)\cup B_{1}(q)$ whenever $s$ is sufficiently close to $t$. To show that $\sigma$ is $C^{1,1}$ with curvature bounded by $1$, it suffices to show that $\sigma$ satisfies a local bilateral ball condition at $\sigma(t)$ for all $t\in [0,1]$.

Let $\intY$ denote the (possibly empty) set of interior points of $Y$. Given $t\in [0,1]$, we denote by $P(t)$ the (possibly empty) set of points in $\pa U$ with distance $1$ from $\sigma(t)$. We remark that $P(t)$ is empty if and only if $\sigma(t)\in \intY$. 

The set of points of $\sigma$ satisfying the bilateral ball condition is compact, since a uniformly bounded sequence of pairs of tangent balls with fixed radius $1$ converges, up to a subsequence, to a pair of tangent balls of radius $1$. Now, consider $t\in [0,1]$ such that $\sigma$ does not satisfy the bilateral ball condition at $x=\sigma(t)$. Then we have the following alternative: either $x\in \intY$, or there exist a unit vector $\nu$ and a constant $c>0$ such that $\big(z-x\big)\cdot \nu \ge c$ for all $z\in P(t)$. 

First suppose that $x\in \intY$. Since $\intY\subset Y_{1}$, by considering local variations of $\sigma$ in an open neighborhood of $x$, we find that $\sigma$ coincides with a segment in this neighborhood. Clearly, $\sigma$ satisfies a local bilateral ball condition at $x$ in this case.

Now consider the latter possibility. Up to a rotation, we may assume $\nu = -e_2$. Then, for $\e>0$ and $|w+e_2|$ small enough, $\dd(x -\e w,\pa U)>1$, so $x -\e w$ belongs to $\intY$. 
In other words, $\intY$ contains the open convex cone 
\[
K_{\theta,\e }(x) = \{z:\ (z-x)\cdot e_2 \ge |z-x|\,\cos\theta\}\cap B_\e(x)
\]
with vertex $x$ for some $0<\theta\le\pi/2$. We now claim that, for $r$ sufficiently small, the smaller cone $K_{\theta/2, \e/2}(y)$ is contained in $\intY$ for all $y \in \pa Y \cap B_r(x).$ Indeed, suppose that this is not the case. 
Then, for some sequence $\{y_i \} \subset \pa Y$ with $y_i \to x$, there exists $z_i \in  K_{\theta/2, \e/2}(y_i) \cap \pa Y$ for each $i$. 
Note that the cones $K_{\theta/2, \e/2}(y_i)$ converge to $K_{\theta/2, \e/2}(x)$ in the Hausdorff distance.
So, up to a subsequence, $z_i$ converges to a limit in $ K_{\theta/2, \e/2}(x) \cap \pa Y$; it follows that this limit is $x$. For each $i$, $Y$ has an exterior tangent ball $B_1(p_i)$ at $z_i$. That is, $z_i \in \pa B_1(p_i)$ and $B_1(p_i)\cap Y$ is empty. Furthermore, since $y_i \not \in B_1(p_i)$, it follows that
\begin{equation}\label{eqn:chunk}
|B_1(p_i) \cap K_{\theta, \e}(y)| \geq c
\end{equation}
for some $c>0$ depending on $\theta$ and $\e$. Up to a subsequence, $p_i \to p$ for some $p$ with $x \in \pa B_1(p)$ and $B_1(p)\cap Y = \emptyset$. On the other hand, passing to the limit in \eqref{eqn:chunk} shows that 
\[
|B_1(p) \cap K_{\theta,\e}(x)| \geq c.
\] 
This is a contradiction as $K_{\theta,\e}(x) $ is contained in $Y$. We have thus proven that the boundary of $Y$ is the graph of a Lipschitz function $\phi$ in a neighborhood $A$ of $x$, so that we can assume $Y$ locally coincides with the closed epigraph of $\phi$. Up to an isometry, we may assume that $x=0$, $A= (-\delta,\delta)\times (-\eta,\eta)$ for $\delta,\eta>0$ small enough, and $\phi:(-\delta,\delta)\to (-\eta,\eta)$. Let us set 
\[
t_1 = \inf\{t\in [0,1]:\ \sigma(t)\in A\}\qquad\text{and}\qquad 
t_2 = \sup\{t\in [0,1]:\ \sigma(t)\in A\}\,.
\]
Of course, $p_i = \sigma(t_i)\in \pa A$ must stay above the graph of $\phi$ for $i=1,2$. We consider the interval $(\alpha_1,\alpha_2)\subset (-\delta,\delta)$ defined by the first coordinates of $p_1$ and $p_2$, without loss of generality assuming that $\alpha_1<\alpha_2$. Notice that $\alpha_1<0<\alpha_2$, otherwise $\sigma$ would not be length minimizing. Furthermore, as $\sigma$ is length minimizing, we infer that $\sigma\cap \overline{A}$ coincides with the graph of the minimal concave function $f:[\alpha_1,\alpha_2]\to [-\eta,\eta]$ such that $p_i = (\alpha_i,f(\alpha_i))$ for $i=1,2$ and $f(u)\ge \phi(u)$ for $u\in (\alpha_1,\alpha_2)$. Finally, we have two possibilities. If $f(0)>\phi(0)$, then $f$ is affine in a neighborhood of $0$. Or, if $f(0)=\phi(0)$, the graph of $f$ satisfies a local bilateral ball condition at $x=(0,f(0))$; indeed, by concavity, the graph of $f$ admits a tangent ball from above at $x$ for any positive radius, while the existence of a tangent ball from below at $x$ with radius $1$ is guaranteed by the fact that $x$ is a boundary point of $Y$ in this case. In conclusion, we have proved that $\sigma$ satisfies a local bilateral ball condition at any point $x=\sigma(t)$ for $0<t<1$. 
\\

\textit{Step two.} We now consider the general case when $\g$ is only continuous. Let $0<\e<r$ be arbitrarily fixed. Since $\g([0,1])$ is compact, we can find a finite partition $0=t_0<t_1<\dots <t_N=1$ of $[0,1]$ with the following property: the piecewise linear curve $\g_\e$, that coincides with $\g$ on $t_i$ for all $i=0,\dots,N$, is contained in the $\e$-neighborhood of $\g$. Consequently, the $(r-\e)$-neighborhood of $\g_\e$ is contained in the $r$-neighborhood of $\g$. Now, since $\g_\e$ has finite length, we apply Step one with $\g_\e$ and $r-\e$ in place of $\g$ and $r$ respectively to obtain a curve $\sigma_\e$ of class $C^{1,1}$ and with curvature bounded by $1/(r-\e)$. Of course, we may assume that $\sigma_{\e}$ is a constant-speed parametrization. 
We are left to prove that the (constant) modulus $v_\e$ of the velocity of $\sigma_\e$ is uniformly bounded in $\e$, so that by Ascoli-Arzel\'a Theorem we obtain a sequence $\{\sigma_{\e_j}\}_j$ converging in $C^{1,1}$ to a parametric curve $\sigma$ with the required properties. In order to show a uniform upper bound on $v_\e$, we argue by contradiction. 
Suppose that, for a sequence $\{\sigma_j =\sigma_{\e_j} \}_j$, the speed $v_{\e_j} = v_j\to +\infty$ as $j\to\infty$. Necessarily, the length of the curve coincides with $v_j$, thus by the boundedness of $\Om$ we can find sequences $\{t_j\}_j$ and $\{s_j\}_j$ in $[0,1]$, such that $|\sigma_{j}(t_j)-\sigma_{j}(s_j)|\to 0$ as $j\to \infty$, but $v_j|t_j-s_j| = 1$ for all $j$. 
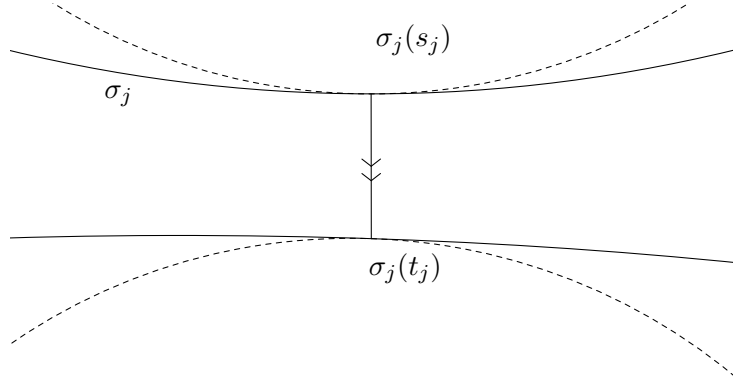
\begin{figure}[t]
\begin{tikzpicture}[line cap=round,line join=round,>=triangle 45,x=4.0cm,y=4.0cm]
\clip(-1.2,1.5) rectangle (1.2,2.8);
\draw [samples=500,rotate around={0.:(0.,2.5)},xshift=0.cm,yshift=10.cm,domain=-20.0:20.0)] plot (\x,{(\x)^2/2/5.0});
\draw [rotate around={0.11345684083991438:(-0.6715922512756595,-0.5730376900197732)}] (-0.6715922512756595,-0.5730376900197732) ellipse (28.262457519307414cm and 10.408642803123875cm);
\draw (0.,2.5)-- (6.186479577519055E-4,2.0184746812432);
\draw (3.7108529729996145E-4,2.211165350456242) -- (-0.030563217611070972,2.23516164183422);
\draw (3.7108529729996145E-4,2.211165350456242) -- (0.031243626887247026,2.2352410492436223);
\draw (3.0932397887609267E-4,2.2592373406216) -- (-0.03062497892949484,2.2832336319995776);
\draw (3.0932397887609267E-4,2.2592373406216) -- (0.031181865568823155,2.2833130394089802);
\draw [dash pattern=on 2pt off 2pt] (0.,4.5) circle (8.cm);
\draw [dash pattern=on 2pt off 2pt] (-0.06631635210308308,0.01959506861845356) circle (8.cm);

\draw (-0.01814645535663145,2.764009422573987) node[anchor=north west] {$\sigma_j(s_j)$};
\draw (-0.04103789813748931,2.0040135222495135) node[anchor=north west] {$\sigma_j(t_j)$};
\draw (-0.920069300922431,2.56256472610244) node[anchor=north west] {$\sigma_j$};

\end{tikzpicture}
\caption{\label{fig:nonoverlapping} A situation occurring in the proof of Theorem~\ref{lem:eraser}}
\end{figure}
Since the restriction of $\sigma_{j}$ to $[s_j,t_j]$ has length $1$ and curvature bounded by $\kappa$, there exists $\tau_j\in (s_j,t_j)$ such that the minimum of the distances of $\sigma_j(\tau_j)$ from $\sigma_{j}(s_j)$ and $\sigma_{j}(t_j)$ is larger than some fixed constant $c>0$ independent of $j$. By the construction of Step one, any point $x\in\sigma_{j}([0,1])$ admits a ball of radius $r_j=r-\e_j$ touching $\sigma$ only at $x$. Therefore, if $j$ is large enough, the segment connecting $\sigma_j(s_j)$ and $\sigma_{j}(t_j)$ must be contained in the inner parallel set $Y_j$ to the $r_j$-neighborhood of $\sigma_j$ (see the definition given in Step one). This is due to the fact that the touching balls at the two points $\sigma_j(s_j)$ and $\sigma_j(t_j)$ do not overlap when $j$ is large (see Figure~\ref{fig:nonoverlapping}). 
At the same time, $Y_j$ is contained in the inner parallel set to the $r_j$-neighborhood of $\g_{\e_j}$. Since the length of such a segment is infinitesimal as $j\to\infty$, we may ``follow the shortcut'' (the arrow line in Figure~\ref{fig:nonoverlapping}) and get a contradiction with the length-minimality of $\sigma_{j}$. This completes the proof of the lemma.
\end{proof}

\section{Proof of Theorem~\ref{theorem:Cheeger formula}}\label{sec: proof of main}
In this section, we prove Theorem~\ref{theorem:Cheeger formula}. We start by proving the following lemma that will be used in the proof.

\begin{lemma}\label{lemma:reach}
Suppose that $\Om \subset \R^2$ is  simply connected and has no necks of radius $r$. Then $\RR(\Om^r) \geq r.$ 
\end{lemma}
\begin{proof}
Recall that $\Om$ having no necks of size $r$ is equivalent to $\Om^r$ being path-connected; see Remark \ref{rmk:nobott}. 
If $\RR(\Om^r) <r$, then there exists $x_0 \in \Om \setminus \Om_r$ 
such that
 $0<\dd(x_0,\Om^r)=t<r$ and $x_0$ has a non-unique projection onto $\Om^r$. Take $y_1,y_2$ to be two distinct points in $\pa B_t(x_0)\cap \pa \Om^r$. 
Let $\ell$ be the line passing through $y_1$ and $y_2$, splitting $\R^2$ into two open half-planes $H^+$ and $H^-$, and let $\pa^\pm \Om= \pa \Om \cap H^\pm$. Since $x_0 \not \in \Om^r$, there exists $y_0 \in \pa \Om$ such that $x_{0} \in B_r(y_0)$. Then, by simple geometric considerations, one finds that $y_0$ cannot be contained in $\ell$, so without loss of generality we can assume $y_0 \in \pa^+\Om$. 

We construct a path from $y_0$ to a point $\bar y \in\pa^- \Om$ that disconnects $\Om^r$. Let $z$ be  the unique point in $H^+$ such that $y_1, y_2 \in \pa B_r(z)$. As $t<r$, the circular arc $\sigma = \pa B_r(z) \cap H^-$ is contained in $B_t(x_0)$. Moreover, $\sigma\cap \Om\neq \emptyset$ and $\sigma\cap \Om^{r} = \emptyset$, otherwise, $\dd(x_0,\Om^r)<t$. 
Let us fix $\bar x \in \sigma\cap \Om$ and show that any projection of $\bar x$ onto $\pa \Om$ lies in $\pa \Om^-$. To this aim, we consider the set $A^{+} = H^{+}\setminus \big(B_{r}(y_{1})\cup B_{r}(y_{2})\big)$. Notice that $\dd(\bar x,A^{+}) = |\bar x -z| = r$, so $\dd(\bar x,\pa^{+}\Om) \ge r$; see Figure \ref{fig:barx}. 
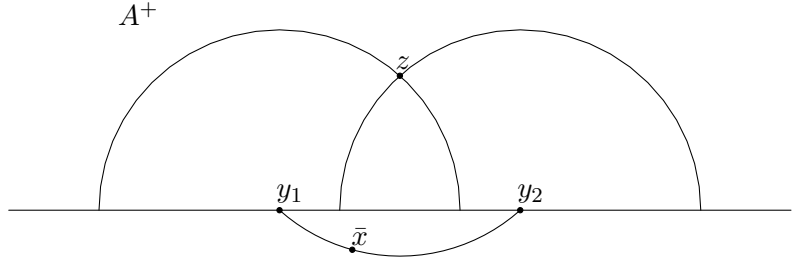
\begin{figure}[t]
\centering
\begin{tikzpicture}[line cap=round,line join=round,>=triangle 45,x=0.8cm,y=0.8cm]
\clip(-6.5,-2.) rectangle (6.5,3.9);
\draw [shift={(2.,0.)}] plot[domain=0.:3.141592653589793,variable=\t]({1.*3.*cos(\t r)+0.*3.*sin(\t r)},{0.*3.*cos(\t r)+1.*3.*sin(\t r)});
\draw [shift={(-2.,0.)}] plot[domain=0.:3.141592653589793,variable=\t]({1.*3.*cos(\t r)+0.*3.*sin(\t r)},{0.*3.*cos(\t r)+1.*3.*sin(\t r)});
\draw [shift={(0.,2.23606797749979)}] plot[domain=3.982661324157723:5.442116636611656,variable=\t]({1.*3.*cos(\t r)+0.*3.*sin(\t r)},{0.*3.*cos(\t r)+1.*3.*sin(\t r)});
\draw [domain=-6.5:6.5] plot(\x,{(--4.-0.*\x)/1.});
\draw (-4.849255632386212,3.6286529456009684) node[anchor=north west] {$A^+$};
\draw [domain=-6.5:6.5] plot(\x,{(-0.-0.*\x)/1.});
\draw [fill=black] (0.,2.23606797749979) circle (1.0pt);
\draw[color=black] (0.05,2.4487846615885984) node {$z$};
\draw [fill=black] (2.,0.) circle (1.0pt);
\draw[color=black] (2.1746477458749385,0.2734025129407902) node {$y_2$};
\draw [fill=black] (-2.,0.) circle (1.0pt);
\draw[color=black] (-1.8258431546045095,0.2734025129407902) node {$y_1$};
\draw [fill=black] (-0.7909905023056598,-0.6577761830006059) circle (1.0pt);
\draw[color=black] (-0.6644103125298311,-0.4455797226292481) node {$\bar{x}$};
\end{tikzpicture}
\caption{A geometric configuration in the proof of Lemma~\ref{lemma:reach} \label{fig:barx}}
\end{figure}
On the other hand, $\dd(\bar x,\pa \Om)<r$. 
Hence, $|\bar x-\bar y| = \dd(\bar x,\pa \Om)$ for some $\bar y\in \pa^{-} \Om$. We now consider the piecewise linear path 
\[
\G = [\bar y, \bar x] \cup [\bar x, x_0] \cup [x_0, y_0]\,.
\]
By construction, $\G \setminus\{y_{0},\bar y\}\subset \Om$ and $\G\cap\Om^{r} = \emptyset$, thus $\G$ disconnects $\Om^{r}$ into two nonempty components, one containing $y_1$ and the other containing $y_2$ (notice indeed that the segment $[\bar x,x_{0}]$ necessarily cuts $[y_{1},y_{2}]$ into two nontrivial subsegments). We reach a contradiction, completing the proof.
\end{proof}

We are now ready to prove our main theorem.

\begin{proof}[Proof of Theorem~\ref{theorem:Cheeger formula}]
Up to rescaling, we may assume that $h(\Om) =1$. Let $E$ be the maximal Cheeger set of $\Om$. By Lemma~\ref{lemma:curvature bound}, $E$ has curvature bounded from above by one, so by Theorem~\ref{theorem:Pestov and Ionin}, $E$ contains a ball of radius one. By Lemma \ref{lem:rollingball} and the assumption that $\Om$ has no necks of radius one, we have
\[
\underset{B_1 \subset \Om}{\bigcup} B_1 \subset E.
\]
We now aim to show the opposite inclusion. 
Let 
\[
G = E \setminus \underset{B_1 \subset \Om}{\bigcup} B_1.
\]
If $G$ is nonempty, choose $x\in G$ and note that $\dd(x,\pa E) <1$. As before, let $P_x$ denote the set of projections of $x$ onto $\pa E$. Take $y \in P_x$ and define $z_t = y + t(x-y)$. For some $t\ge 1$, $z_t$ is a cut point with $y\in P_{z_t}$. Furthermore, $\dd(z_t, \pa E)<1$, otherwise $x$ would belong to the union of balls of radius $1$ contained in $\Om$. Hence, by Lemma~\ref{prop: focal point distance}, $z_t$ is not a focal point of $E$, so $\#P_{z_t} >1$ by Lemma~\ref{prop: at least one}.
Let $\g$ be a maximal path (with respect to inclusion) in 
{$\mathcal{C}\cap \{a\in E:\dd(a,\pa E)<1\}$} containing $z_t$; such a path exists by Zorn's Lemma and is defined, say, on a bounded open interval $(a,b)$. By Proposition~\ref{lem: structure}, $\g$ is not reduced to the single point $z_t$. We now split the proof in two cases.

\textit{Case one: the endpoints $\g(a)$ and $\g(b)$ of the curve $\g$ are well-defined in the limit sense.} Arguing as in the proof of Theorem~\ref{theorem:Pestov and Ionin}, we determine that the closure of $\g$ is not a loop. Moreover, $\dd(\g(a),\pa E)\ge 1$ and $\dd(\g(b),\pa E)\ge 1$. Since $\Om$ has no necks of radius one, there exists a path $\tilde{\g}$ with endpoints $\g(a)$ and $\g(b)$ such that $B_1(z) \subset \Omega$ for all $z \in \tilde{\g}$. Lemma \ref{lem:rollingball} ensures that these balls are contained in $E$ as well. Now, consider the closed loop $\sigma$ obtained by concatenation of the two paths $\g$ and $\tilde \g$. Notice that $\sigma$ is a simple loop as $\g$ and $\tilde\g$ do not intersect except that at the endpoints. Since $E$ is simply connected, the domain $D_{\sigma}$ bounded by $\sigma$ is compactly contained in $E$. 
Furthermore, since $\g$ is piecewise $C^{1}$ and of positive length (Proposition~\ref{lem: structure}), almost all points $z\in \g$ have $\#P_{z} = 2$. Let us fix one such point $x\in \g$, so $\#P_{x}=2$ and $\dd(x,\pa E)<1$. The two segments $[y_i, x]$ for $y_i \in P_x$, $i=1,2$, are transversal to the tangent to $\g$ at $x$, and lie on opposite sides of the tangent line to $\g$ at $x$. 
Hence, one of the segments has nonempty intersection with the interior of $D_\sigma$. Suppose it is the first one, $[y_1,x]$. Since $y_1 \in \pa E$, the segment $[y_1,x]$ must intersect $\sigma$ at some $x'\neq x$ with $\#P_{x'} =1$. However, $x' \not \in \tilde{\g}$, as this would imply that $\dd(x,\pa E) \geq 1$. while $x' \in \g$, as this implies $\#P_{x'} \geq2.$ We reach a contradiction, concluding the proof of \eqref{eqn: union} in this case.

\textit{Case two: the endpoints of $\g$ are not well-defined.} 
In this case, we replace $\g$ with another curve $\hat\g$ obtained in the following way. First, we restrict $\g$ to a sequence of compact subintervals $[\alpha_{j},\beta_{j}]\subset (a,b)$ such that $\alpha_{j}\to a$, $\beta_{j}\to b$, $\g(\alpha_{j})\to z_{a}$ and $\g(\beta_{j})\to z_{b}$ as $j\to\infty$ for some $z_a, z_b\in \Om$. Therefore by the maximality of $\g$ and owing to Propositon~\ref{lem: structure}, we obtain that 
{$\dd(z_{a},\pa E)=\dd(z_{b},\pa E)= 1$}. Consequently, we may assume the existence of $j\in \N$ and $t\in (\alpha_{j},\beta_{j})$ such that, setting $x=\g(t)$, we have $\# P_{x} = 2$ and 
\[
\dd(x,\pa E) < \min \{\dd(\g(\alpha_{j}),\pa E),\ \dd(\g(\beta_{j}),\pa E)\}\,.
\]
Third, by connecting $\g(\alpha_{j})$ and $\g(\beta_{j})$ to, respectively, $z_{a}$ and $z_{b}$ with two straight segments, we would obtain $\hat\g$ as a replacement of $\g$, having $z_{a}$ and $z_{b}$ as endpoints. Up to choosing $j$ large enough, we can also assume that $\dd(x,\pa E) < \dd(y,\pa E)$ for every $y$ belonging to each straight segment. Then we can repeat the same proof as in Case one, with $\hat\g$ in place of $\g$.

Now we show that \eqref{eqn: r} holds true as well. Thanks to \eqref{NB} and Remark \ref{rmk:nobott}, $\Om^{r}$ is connected. Then Lemma~\ref{lemma:reach} implies that $\RR(\Om^r ) \geq r$. We can apply the Steiner formulas \eqref{eqn: Steiner2} to $E = \Om^r \oplus B_{r}$. Since 
\[
\frac{1}{r} = h(\Om)= \frac{P(E)}{|E|},
\]
 we deduce that
\[
 \M_0(\Om^r)r + 2\pi r^2 = r P(E) = |E|  = |\Om^r|+ \M_0(\Om^r)r + \pi r^2.
\]
That is,
\[
|\Om^r| = \pi r^2.
\]
\end{proof}

We now prove Corollary~\ref{cor: cor2}. 

\begin{proof}[Proof of Corollary~\ref{cor: cor2}]
We obtain the following simple bounds above and below on the Cheeger constant of $\Omega$:
\[
\frac{n\om_n^{1/n}}{|\Om|^{1/n}}\leq h(\Om) \leq \frac{2}{\inr(\Om)}.
\]
Indeed, the bound below comes from applying the isoperimetric inequality to any $E\subset \Om$ and using $|E| \leq |\Om|$, while the upper bound follows simply by taking the competitor $E=$ a ball contained in $\Om$ with the largest possible radius in the minimization problem. Hence, the assumption \eqref{NB'} on $\Om$ in particular implies that $\Om$ has no necks of radius $r =1/h(\Om)$. Applying Theorem~\ref{theorem:Cheeger formula}, we conclude the proof.
\end{proof}

\begin{remark}\label{rem:uniqueness}\rm 
Notice that whenever $\overline{\text{\rm int}(\Omega^r)} = \Omega^r$, then the maximal Cheeger set of $\Om$ is also minimal, i.e. the Cheeger set is unique. To see this one can argue by contradiction. Suppose that $E_{{\rm max}} \setminus E_{\rm min} \neq \emptyset$ and take a point $x$ in the difference. Then, by  \eqref{eqn: union} there exists a ball $B_r \subset E_{{\rm max}}$ containing $x$, hence by $\overline{\text{\rm int}(\Omega^r)} = \Omega^r$ we can find $y\in {\rm int}(\Omega^r)$ such that $x\in B_r(y)$. By Theorem \ref{theorem:Pestov and Ionin}, we can find a ball $\widetilde{B}_r(z)\subset E_{\rm min}$, then by the no neck assumption there exists a continuous path $\g$ connecting $z$ with $y$, such that $B_r(\g(t)) \subset E_{\rm max}$. Owing to Theorem~\ref{lem:eraser} we can directly assume $\g$ of class $C^{1,1}$ and with curvature bounded by $h(\Om) = 1/r$. Let $\bar{t}<1$ be the last time such that $B_r(\gamma(t))\subset E_{\rm min}$. Of course we have $\g(\bar{t}) \neq \g(1)=y$. 
Then Lemma \ref{lem:rollingball} ensures that $U_\tau = E_{\rm min} \cup \bigcup_{t\in [\bar t,\tau]}$ is a Cheeger set of $\Om$ for all $\tau\ge \bar t$. This, however, leads to a contradiction because $\pa U_1 \cap \pa B_r(y)$ is an arc of radius $r$ whose endpoints do not belong to $\pa \Om$. 
\end{remark}

\section{Computation of the Cheeger constant of Koch's snowflake}\label{sec:Koch}

As an application of Theorem \ref{theorem:Cheeger formula}, we find an efficient procedure for the approximation of the Cheeger constant of a Koch's snowflake $K$. Before describing the procedure, in the next lemma we prove a general error estimate involving the inverse of the Cheeger constants of two domains $\om\subset \Om$ for which the formula \eqref{eqn: r} holds. This lemma will be crucial for estimating $h(K)$ with increasing precision.
\begin{lemma}\label{lem:approx}
	Let $\om,\Om$ be two planar domains such that $\om\subset \Om$. Let us set $r = h(\om)^{-1}$ and $R = h(\Om)^{-1}$ and assume that \eqref{eqn: r} holds for both, that is, $|\om^r| = \pi r^2$ and $|\Om^R| = \pi R^2$. Then 
	\begin{equation}\label{eq:errorestimate}
	0\le R - r \le \frac{|\Om^{r}\setminus \om^{r}|}{2\pi r}\,.
	\end{equation}
\end{lemma}
\begin{proof}
The first inequality in \eqref{eq:errorestimate} directly follows from the monotonicity of the Cheeger constant with respect to inclusion, that is, $h(\Om) \le h(\om)$. In order to prove the second inequality we observe that
\[
\pi R^2 = |\Om^{R}| \le |\Om^{r}| = |\om^{r}| + |\Om^{r}\setminus \om^{r}|\,,
\]
whence
\[
R \le \sqrt{|\Om^{r}| / \pi} = r \sqrt{|\Om^{r}|/|\om^{r}|} = r\sqrt{1+ \frac{|\Om^{r} \setminus \om^{r}|}{\pi r^2}}\,.
\]
Using the elementary inequality $\sqrt{1+t} \le 1+\frac t2$ we get
\[
R \le r\left(1 + \frac{|\Om^{r} \setminus \om^{r}|}{2\pi r^2}\right)
\]
and \eqref{eq:errorestimate} follows at once.
\end{proof}

It is convenient to start the construction of Koch's snowflake $K$ from an equilateral triangle $K_1$ with side length equal to $3$. We let $K_n$ be the $n$-th step in the construction of $K$, and recall that $K_{n+1}$ is the polygon obtained by attaching $3\cdot 4^{n-1}$ equilateral triangles of side length $3^{1-n}$ to the middle thirds of the sides of $K_n$. 

First, it is obvious that $K_n$ satisfies \eqref{T}, while it is not difficult to prove that  \eqref{NB} holds for all $n\in \N$. The idea is the following: given any pair of balls of equal radius that are contained in $K_n$, by the symmetry of $K_n$ we can move each ball, without exiting $K_n$, in such a way that its center is translated onto an axis of symmetry of a smallest triangle (one of those used in the construction of $K_n$) that has a non-empty intersection with the ball. Then we can similarly move the center (and the ball) along the axis until it reaches another axis of symmetry of a triangle of bigger size. By iteratively repeating this procedure, we finally move each ball inside $K_n$ until its center coincides with the center of symmetry of the initial triangle $K_1$. This shows that, for all $n\in \N$, $K_n$ satisfies the no-neck property (see Definition \ref{def:noneck}) for any radius $r>0$ smaller than the inradius of $K_n$.

Since the snowflake $K$ is the union of all $K_n$, it necessarily satisfies \eqref{T} and \eqref{NB}. Then, thanks to Theorem \ref{theorem:Cheeger formula}, the formula \eqref{eqn: r} holds with $K$ (respectively, $K_n$) in place of $\Om$ and $r=h(K)^{-1}$ (respectively, $r=r_n=h(K_n)^{-1}$). We first note that, by Lemma \ref{lem:approx}, we have 
\begin{equation*}
0\le r - r_n \le \frac{|K^{r_n}\setminus K_n^{r_n}|}{2\pi r_n}\,.
\end{equation*} 
By \eqref{eqn: r}, one can easily compute the Cheeger constant of the equilateral triangle $K_1$ to be 
\[
h(K_1) = \frac{6 + 2\sqrt{\pi\sqrt{3}}}{3\sqrt{3}}\,.
\]
Let us now compute $h(K_2)$. First, we note that $h(K_2)\le 2$ as the inradius of $K_2$ is $1$. In the following, we let $r = r_2 = h(K_2)^{-1}$ and distinguish between two cases.

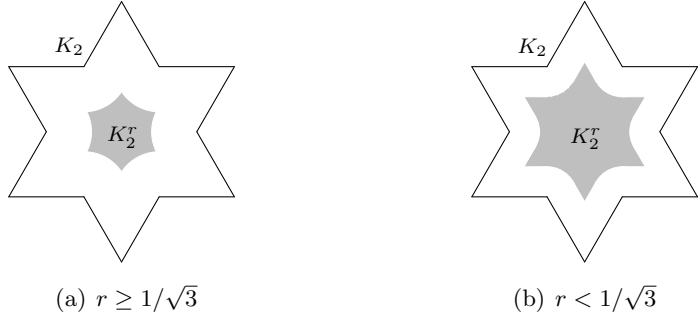
\begin{figure}[t]%
	\centering
	\definecolor{ffffff}{rgb}{1,1,1}
	\definecolor{cqcqcq}{rgb}{0.75,0.75,0.75}
	\subfigure[$r\ge 1/\sqrt{3}$\label{K2rlarge}]{\begin{tikzpicture}[line cap=round,line join=round,>=triangle 45,x=1.0cm,y=1.0cm]
		\clip(-0.2,-1.) rectangle (3.2,2.7);
		\fill[color=cqcqcq,fill=cqcqcq,fill opacity=1.0] (1.5,2.5980762113533165) -- (1.,1.7320508075688767) -- (0.,1.7320508075688767) -- (0.5,0.8660254037844379) -- (0.,0.) -- (1.,0.) -- (1.5,-0.8660254037844397) -- (2.,0.) -- (3.,0.) -- (2.5,0.8660254037844388) -- (3.,1.7320508075688767) -- (2.,1.7320508075688767) -- cycle;
		\draw [color=ffffff,fill=ffffff,fill opacity=1.0] (0.5,0.8660254037844379) circle (0.6cm);
		\draw [color=ffffff,fill=ffffff,fill opacity=1.0] (1.,1.7320508075688767) circle (0.6cm);
		\draw [color=ffffff,fill=ffffff,fill opacity=1.0] (2.,1.7320508075688767) circle (0.6cm);
		\draw [color=ffffff,fill=ffffff,fill opacity=1.0] (2.5,0.8660254037844388) circle (0.6cm);
		\draw [color=ffffff,fill=ffffff,fill opacity=1.0] (2.,0.) circle (0.6cm);
		\draw [color=ffffff,fill=ffffff,fill opacity=1.0] (1.,0.) circle (0.6cm);
		\draw [color=ffffff,fill=ffffff,fill opacity=1.0] (0.,1.7320508075688767) circle (0.6cm);
		\draw [color=ffffff,fill=ffffff,fill opacity=1.0] (1.5,2.5980762113533165) circle (0.6cm);
		\draw [color=ffffff,fill=ffffff,fill opacity=1.0] (3.,1.7320508075688767) circle (0.6cm);
		\draw [color=ffffff,fill=ffffff,fill opacity=1.0] (3.,0.) circle (0.6cm);
		\draw [color=ffffff,fill=ffffff,fill opacity=1.0] (1.5,-0.8660254037844397) circle (0.6cm);
		\draw [color=ffffff,fill=ffffff,fill opacity=1.0] (0.,0.) circle (0.6cm);
		\draw (1.5,2.5980762113533165)-- (1.,1.7320508075688767);
		\draw (1.,1.7320508075688767)-- (0.,1.7320508075688767);
		\draw (0.,1.7320508075688767)-- (0.5,0.8660254037844379);
		\draw (0.5,0.8660254037844379)-- (0.,0.);
		\draw (0.,0.)-- (1.,0.);
		\draw (1.,0.)-- (1.5,-0.8660254037844397);
		\draw (1.5,-0.8660254037844397)-- (2.,0.);
		\draw (2.,0.)-- (3.,0.);
		\draw (3.,0.)-- (2.5,0.8660254037844388);
		\draw (2.5,0.8660254037844388)-- (3.,1.7320508075688767);
		\draw (3.,1.7320508075688767)-- (2.,1.7320508075688767);
		\draw (2.,1.7320508075688767)-- (1.5,2.5980762113533165);
		
		\begin{scriptsize}
		\draw (1.2,1.0369646945610889) node[anchor=north west] {$K_2^r$};
		\draw (0.5,2.2217881673665136) node[anchor=north west] {$K_2$};
		\end{scriptsize}
		\end{tikzpicture}
	}%
	\hspace{2.5cm}
	\subfigure[\label{K2rsmall}$r<1/\sqrt{3}$]
	{\begin{tikzpicture}[line cap=round,line join=round,>=triangle 45,x=1.0cm,y=1.0cm]
		\clip(-0.2,-1.) rectangle (3.2,2.7);
		\fill[dotted,color=cqcqcq,fill=cqcqcq,fill opacity=1.0] (1.5,2.5980762113533165) -- (1.,1.7320508075688767) -- (0.,1.7320508075688767) -- (0.5,0.8660254037844379) -- (0.,0.) -- (1.,0.) -- (1.5,-0.8660254037844397) -- (2.,0.) -- (3.,0.) -- (2.5,0.8660254037844388) -- (3.,1.7320508075688767) -- (2.,1.7320508075688767) -- cycle;
		\draw [color=ffffff,fill=ffffff,fill opacity=1.0] (0.5,0.8660254037844379) circle (0.4cm);
		\draw [color=ffffff,fill=ffffff,fill opacity=1.0] (0.25,1.2990381056766573) circle (0.4cm);
		\draw [color=ffffff,fill=ffffff,fill opacity=1.0] (0.25,0.43301270189221924) circle (0.4cm);
		\draw [color=ffffff,fill=ffffff,fill opacity=1.0] (1.,0.) circle (0.4cm);
		\draw [color=ffffff,fill=ffffff,fill opacity=1.0] (1.5,2.5980762113533165) circle (0.4cm);
		\draw [color=ffffff,fill=ffffff,fill opacity=1.0] (3.,1.7320508075688767) circle (0.4cm);
		\draw [color=ffffff,fill=ffffff,fill opacity=1.0] (1.75,-0.4330127018922198) circle (0.4cm);
		\draw [color=ffffff,fill=ffffff,fill opacity=1.0] (3.,0.) circle (0.4cm);
		\draw [color=ffffff,fill=ffffff,fill opacity=1.0] (1.5,-0.8660254037844397) circle (0.4cm);
		\draw [color=ffffff,fill=ffffff,fill opacity=1.0] (2.5,0.) circle (0.4cm);
		\draw [color=ffffff,fill=ffffff,fill opacity=1.0] (0.,0.) circle (0.4cm);
		\draw [color=ffffff,fill=ffffff,fill opacity=1.0] (2.75,0.4330127018922194) circle (0.4cm);
		\draw [color=ffffff,fill=ffffff,fill opacity=1.0] (0.,1.7320508075688767) circle (0.4cm);
		\draw [color=ffffff,fill=ffffff,fill opacity=1.0] (2.75,1.2990381056766573) circle (0.4cm);
		\draw [dotted,color=ffffff,fill=ffffff,fill opacity=1.0] (1.,1.7320508075688767) circle (0.4cm);
		\draw [color=ffffff,fill=ffffff,fill opacity=1.0] (0.5,1.732050807568877) circle (0.4cm);
		\draw [color=ffffff,fill=ffffff,fill opacity=1.0] (2.5,1.7320508075688772) circle (0.4cm);
		\draw [color=ffffff,fill=ffffff,fill opacity=1.0] (2.,1.7320508075688767) circle (0.4cm);
		\draw [color=ffffff,fill=ffffff,fill opacity=1.0] (1.75,2.165063509461094) circle (0.4cm);
		\draw [color=ffffff,fill=ffffff,fill opacity=1.0] (2.5,0.8660254037844388) circle (0.4cm);
		\draw [color=ffffff,fill=ffffff,fill opacity=1.0] (1.25,2.1650635094610964) circle (0.4cm);
		\draw [color=ffffff,fill=ffffff,fill opacity=1.0] (2.,0.) circle (0.4cm);
		\draw [color=ffffff,fill=ffffff,fill opacity=1.0] (0.5,0.) circle (0.4cm);
		\draw [color=ffffff,fill=ffffff,fill opacity=1.0] (1.25,-0.4330127018922175) circle (0.4cm);
		\fill[color=ffffff,fill=ffffff,fill opacity=1.0] (0.5,1.3320508075688766) -- (1.,1.3320508075688766) -- (1.3464101615137751,1.5320508075688766) -- (1.5964101615137753,1.9650635094610964) -- cycle;
		\fill[color=ffffff,fill=ffffff,fill opacity=1.0] (2.5,1.3320508075688768) -- (2.,1.3320508075688766) -- (1.6535898384862255,1.5320508075688766) -- (1.4035898384862253,1.9650635094610964) -- cycle;
		\fill[color=ffffff,fill=ffffff,fill opacity=1.0] (2.5,0.4) -- (2.,0.4) -- (1.6535898384862262,0.2) -- (1.4035898384862266,-0.23301270189221934) -- cycle;
		\fill[color=ffffff,fill=ffffff,fill opacity=1.0] (0.5,0.4) -- (1.,0.4) -- (1.3464101615137747,0.2) -- (1.5964101615137742,-0.23301270189221934) -- cycle;
		\fill[color=ffffff,fill=ffffff,fill opacity=1.0] (0.5964101615137771,0.23301270189221848) -- (0.8464101615137769,0.6660254037844371) -- (0.8464101615137767,1.066025403784437) -- (0.596410161513776,1.4990381056766564) -- cycle;
		\fill[color=ffffff,fill=ffffff,fill opacity=1.0] (2.4035898384862238,0.2330127018922186) -- (2.1535898384862238,0.6660254037844372) -- (2.153589838486224,1.066025403784437) -- (2.4035898384862247,1.4990381056766564) -- cycle;
		\draw (1.5,2.5980762113533165)-- (1.,1.7320508075688767);
		\draw (1.,1.7320508075688767)-- (0.,1.7320508075688767);
		\draw (0.,1.7320508075688767)-- (0.5,0.8660254037844379);
		\draw (0.5,0.8660254037844379)-- (0.,0.);
		\draw (0.,0.)-- (1.,0.);
		\draw (1.,0.)-- (1.5,-0.8660254037844397);
		\draw (1.5,-0.8660254037844397)-- (2.,0.);
		\draw (2.,0.)-- (3.,0.);
		\draw (3.,0.)-- (2.5,0.8660254037844388);
		\draw (2.5,0.8660254037844388)-- (3.,1.7320508075688767);
		\draw (3.,1.7320508075688767)-- (2.,1.7320508075688767);
		\draw (2.,1.7320508075688767)-- (1.5,2.5980762113533165);
		\draw [color=ffffff] (0.5,1.3320508075688766)-- (1.,1.3320508075688766);
		\draw [color=ffffff] (1.,1.3320508075688766)-- (1.3464101615137751,1.5320508075688766);
		\draw [color=ffffff] (1.3464101615137751,1.5320508075688766)-- (1.5964101615137753,1.9650635094610964);
		\draw [color=ffffff] (1.5964101615137753,1.9650635094610964)-- (0.5,1.3320508075688766);
		\draw [color=ffffff] (2.5,1.3320508075688768)-- (2.,1.3320508075688766);
		\draw [color=ffffff] (2.,1.3320508075688766)-- (1.6535898384862255,1.5320508075688766);
		\draw [color=ffffff] (1.6535898384862255,1.5320508075688766)-- (1.4035898384862253,1.9650635094610964);
		\draw [color=ffffff] (1.4035898384862253,1.9650635094610964)-- (2.5,1.3320508075688768);
		\draw [color=ffffff] (2.5,0.4)-- (2.,0.4);
		\draw [color=ffffff] (2.,0.4)-- (1.6535898384862262,0.2);
		\draw [color=ffffff] (1.6535898384862262,0.2)-- (1.4035898384862266,-0.23301270189221934);
		\draw [color=ffffff] (1.4035898384862266,-0.23301270189221934)-- (2.5,0.4);
		\draw [color=ffffff] (0.5,0.4)-- (1.,0.4);
		\draw [color=ffffff] (1.,0.4)-- (1.3464101615137747,0.2);
		\draw [color=ffffff] (1.3464101615137747,0.2)-- (1.5964101615137742,-0.23301270189221934);
		\draw [color=ffffff] (1.5964101615137742,-0.23301270189221934)-- (0.5,0.4);
		\draw [color=ffffff] (0.5964101615137771,0.23301270189221848)-- (0.8464101615137769,0.6660254037844371);
		\draw [color=ffffff] (0.8464101615137769,0.6660254037844371)-- (0.8464101615137767,1.066025403784437);
		\draw [color=ffffff] (0.8464101615137767,1.066025403784437)-- (0.596410161513776,1.4990381056766564);
		\draw [color=ffffff] (0.596410161513776,1.4990381056766564)-- (0.5964101615137771,0.23301270189221848);
		\draw [color=ffffff] (2.4035898384862238,0.2330127018922186)-- (2.1535898384862238,0.6660254037844372);
		\draw [color=ffffff] (2.1535898384862238,0.6660254037844372)-- (2.153589838486224,1.066025403784437);
		\draw [color=ffffff] (2.153589838486224,1.066025403784437)-- (2.4035898384862247,1.4990381056766564);
		\draw [color=ffffff] (2.4035898384862247,1.4990381056766564)-- (2.4035898384862238,0.2330127018922186);
		\begin{scriptsize}
		\draw (1.2,1.0160286023398883) node[anchor=north west] {$K_2^r$};
		\draw (0.5,2.208912242071308) node[anchor=north west] {$K_2$};
		\end{scriptsize}
		\end{tikzpicture}}
	\caption{The above figures display the different shapes that the inner retract $K_2^r$ can assume.}
	\label{tutta}
\end{figure}

\textit{Case one: $r\ge 1/\sqrt{3}$.} In this case the set $K_2^r = \{x\in K_2:\ \dd(x,\pa K_2)\ge r\}$ is a hexagon with six circular arcs as its sides (see Figure \ref{K2rlarge}). Its area is
{
\[
|K_2^r|= \pi r^2 +\frac{3\sqrt{3}}{2}- 3r\cos \alpha -6\alpha r^2\,,
\]
}
where $\alpha = \arcsin\left(\frac{1}{2 r}\right)$.
{Then, observing that the decreasing function	
\[
|K_2^r| - \pi r^2 = \frac{3\sqrt{3}}{2}- 3r\cos \alpha -6\alpha r^2
\]
is negative when $r=1/\sqrt{3}$, we deduce that} 
formula \eqref{eqn: r}
{cannot hold for $r\in [1/\sqrt{3},1]$}. This shows that case one is impossible.

\textit{Case two: $1/2 \le r< 1/\sqrt{3}$.} In this case, $K_2^r$ still resembles an hexagon with curved sides, but now each side comprises a circular arc with two segments of equal length attached to it (see Figure \ref{K2rsmall}). After some elementary computations, we find that the area of $K_2^r$ is
\[
|K_2^r| = (6\sqrt{3} -\pi)r^2 -12 r +3\sqrt{3}\,.
\]
The equation $|K_2^r| = \pi r^2$ has two solutions: one larger than $1$ (and thus to be excluded), the other being
\[
r_2 = \frac{6-\sqrt{6\pi\sqrt{3}-18}}{6\sqrt{3} -2\pi} = 0.5287455502\dots
\]
(the displayed digits are exact). Therefore, the Cheeger constant of $K_2$ is 
\[
h(K_2) = r_2^{-1} = \frac{6\sqrt{3} -2\pi}{6-\sqrt{6\pi\sqrt{3}-18}} = 1.8912688715\dots
\]
Let us now consider $K_n$ for $n\ge 2$. Since $K_2\subset K_n$, we have  $h(K_n)\le h(K_2)$ and thus $r_n\ge r_2$. By Theorem \ref{theorem:Cheeger formula}, the Cheeger set $E_n$ of $K_n$ is the union of balls of radius $r_n$ that are contained in $K_n$. 

We note that $E_n$ is contained in the union of all balls of radius $r_2$ contained in $K_n$. Denoting this set by $G_n$ and observing that  $r_2 > 8/9\tan(\pi/6)$, we see from the construction that $G_n \subset K_2$ for $n=3,4$. So, $E_n\subset K_2$, hence $E_n = E_2$, when $n=3,4$. 

\begin{figure}[t]
\begin{tikzpicture}[line cap=round,line join=round,>=triangle 45,x=1.0cm,y=1.0cm]
\clip(-4.2,-2.) rectangle (4.2,5.3);
\draw (4.,0.)-- (3.,0.)-- (2.6653885316663746,0.5795640639490639)-- (3.0034705488670235,1.1637175695193298)-- (2.3334203644083398,1.154549796135504)-- (2.,1.7320508075688774)-- (2.3307855315040102,2.304988154542497)-- (2.9980518680355117,2.307059224967802)-- (2.6659326114845205,2.8854799250770906)-- (3.,3.4641016151377526)-- (2.3369728065913726,3.4641016151377526)-- (1.9955959142062836,4.062711814376829)-- (1.6542190218211952,3.4641016151377526)-- (1.,3.464101615137752)-- (0.663845813226438,4.046337745806564)-- (1.0159866000869044,4.619473602403172)-- (0.33102400844293056,4.62280202155837)-- (0.,5.196152422706633);
\draw (-4.,0.)-- (-3.,0.)-- (-2.6653885316663746,0.5795640639490642)-- (-3.0034705488670235,1.1637175695193303)-- (-2.3334203644083398,1.1545497961355042)-- (-2.,1.7320508075688776)-- (-2.33078553150401,2.3049881545424973)-- (-2.9980518680355113,2.3070592249678024)-- (-2.66593261148452,2.885479925077091)-- (-3.,3.464101615137753)-- (-2.336972806591372,3.464101615137753)-- (-1.9955959142062831,4.062711814376829)-- (-1.6542190218211947,3.4641016151377526)-- (-1.,3.464101615137752)-- (-0.6638458132264375,4.046337745806564)-- (-1.0159866000869038,4.619473602403172)-- (-0.33102400844293,4.62280202155837)-- (0.,5.196152422706633);
\draw [line width=0.4pt,dashed] (0.,5.196152422706633)-- (0.,-2.);
\draw (-1.0210215874367414,-2.468879493480799)-- (0.,-0.7004182284157352);
\draw (1.0210215874367412,-2.468879493480799)-- (0.,-0.7004182284157352);
\draw [line width=0.4pt,dotted] (0.,-0.7004182284157352)-- (2.55328998954086,0.7737244343648563);
\draw [line width=0.4pt,dotted] (0.,5.196152422706633)-- (2.988151054632557,0.020522975392519762);
\draw (-0.20686450907446513,1.7364341133009897) node {$\ell$};
\draw (-0.3,-0.6) node {$Q_x$};
\draw (0.25,5.2) node {$V$};
\draw (2.50672816167002,0.4202877149329806) node {$P_x$};
\draw (3.2,0.2) node {$W$};
\end{tikzpicture}
\caption{A zoomed-in detail of $K_n$ with the notation used throughout Section \ref{sec:Koch}.\label{fig:QxPx}}
\end{figure}
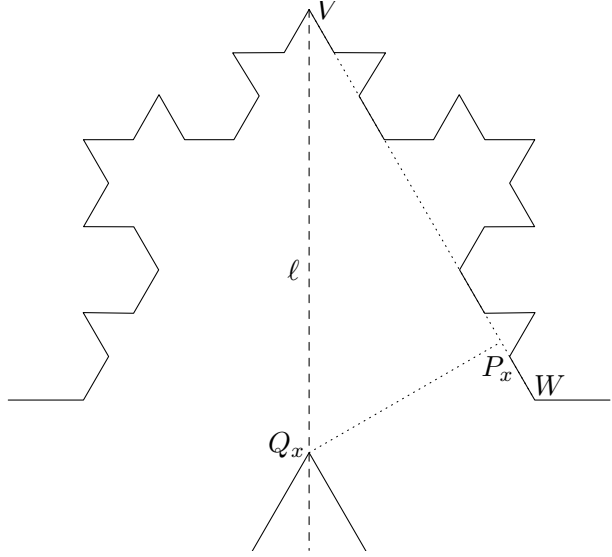

On the other hand, when we consider $K_n$ for $n\ge 5$, the situation changes. 
It will be convenient to express the inner retract $K_n^r$ in terms of a parameter $x$. With reference to Figure \ref{fig:QxPx}, fix a vertex $V$ of $K_2$ and for $x\in [0,1]$, let $P_x$ be the point on a side of $K_2$ adjacent to $V$ with $|V-P_x|=x.$ Correspondingly, $x$ determines a unique radius $r$ such that $|Q_x-V|=2x/\sqrt{3} $ where $Q_x$ is the point minimizing $|Q-V|$ among points $Q$ with $B_r(Q)\subset K_n$. Note that $P_x$ is then the projection of $Q_x$ onto the side of $K_2$ and $|Q_x-P_x| = x/\sqrt{3}$; again see Figure \ref{fig:QxPx}. When $r = r_n$, we denote by $x_n$ the corresponding parameter.

The shape of the inner retract $K_n^r$ may be different according to the value of the parameter $x=x_n$ (from which, as noted before, the value of $r$ can be obtained). Reasoning as above, we see that $x_n \in (8/9,1)$ for $n\geq 5$. A formula for the area of $K_n^r$ as a function of $x\in (8/9,1)$ can be written by first observing that $|K_n^r| = |K_2^r|+|K_n^r \setminus K_2^r|$. Then, the problem is reduced to characterizing the set $K_n^r \setminus K_2^r$. To this aim, we recall that the inequality $r_n >1/2$ forces the Cheeger set $E_n$ to only intersect the equilateral triangles that have one side on the boundary of $K_2$. Correspondingly, $K_n^r\setminus K_2^r$ results in a disjoint union of the parametric shapes depicted in Figure \ref{fig:InnerCasi23}. Thanks to the symmetry of $K_2$, it suffices to focus on $1/12$-th of the set $K_n^r\setminus K_2^r$. We now compute the area of the shapes comprising $K_n^r\setminus K_2^r$, and consisting of (portions of) curvilinear isosceles triangles, with concave circular arcs of radius $r$ replacing the two equal sides (see Figure \ref{fig:InnerCasi23}). We start by defining
\[
\cI(r,\beta) = \beta r  - \frac{\beta}{2}\sqrt{r^2 - \frac{\beta^2}{4}} -r^2 \arcsin \left(\frac{\beta}{2r} \right)\,,
\]
which represents the area of a curvilinear isosceles triangle with base of length $\beta$ and curved equal sides with curvature $1/r$ tangentially touching the base, as depicted in Figure \ref{fig:caso1inner}. We note for future reference the following estimate:
\begin{equation}\label{eq:stimaIrbeta}
\cI(r,\beta) \le \frac{\beta^3}{12 r}\,.
\end{equation}
The proof of \eqref{eq:stimaIrbeta} simply follows from the fact that the derivative of the function $\phi(t) = 2t-t\sqrt{1-t^2} -\arcsin(t)$ satisfies $\phi'(t) = 2-2\sqrt{1-t^2} <2t^2$, whence $\phi(t) \le 2t^3/3$ for $0\le t \le 1$. 

We also need to compute the areas of the portions of a curvilinear isosceles triangle, which are conveniently defined in terms of the variable $x$ and the distance from $C_n$, the $n$th iteration in the construction of the Cantor set (see Figures \ref{fig:caso2inner} and \ref{fig:caso3inner}). 
Let us set the following parameters:
\begin{itemize}
	\item $d = \dd(x,C_n)$;
	\item $r = \sqrt{x^2/3 + d^2}$;
	\item $\iota=\iota(x)$ as the largest element of $C_n$ such that $\iota\le x$;
	\item $\beta=\beta(x)$ as the length of the connected component of $\overline{[0,1]\setminus C_n}$ that contains $x$ (if there is no such component, set $\beta=0$). 
\end{itemize} 
If $x=d+\iota$, we let $A(x)$ be the area of the curvilinear polygon depicted in Figure \ref{fig:caso2inner} and have  
\begin{align*} \label{eq:caso3K}
\nonumber
A(x) = \frac{1}{2\sqrt{3}} &\left( xd +(\beta-d)  (2x-d+\beta) - (x-d+\beta -r\sqrt{3})^2 \right) \\
&-\frac{\beta}{2}\sqrt{r^2 - \frac{\beta^2}{4}} -\frac{r^2}{2} \left[2\arcsin \left( \frac{\beta}{2r}\right) -\arcsin \left( \frac{d}{r}\right) \right]\,. 
\end{align*}
Similarly, if $x=\beta+\iota-d$, then denoting by $B(x)$ the area of the curvilinear triangle depicted in Figure \ref{fig:caso3inner}, we have
\begin{equation*}\label{eq:caso2K}
B(x) = (x + d)r - \left(\frac{2x^2 +3d^2 +xd}{2\sqrt{3}} + \frac{r^2}{2} \arcsin (d/r) \right)\,.
\end{equation*}

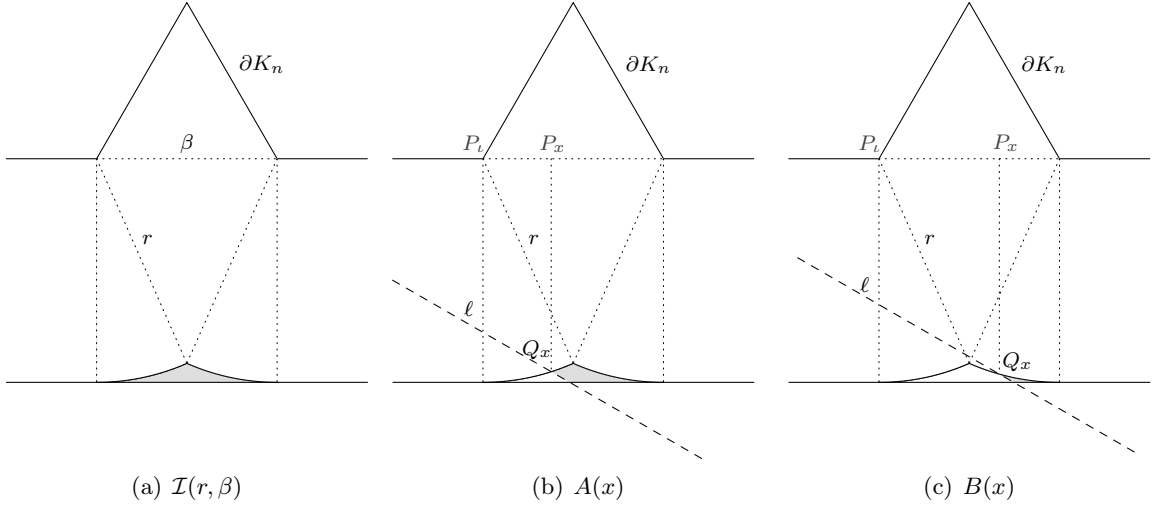
\begin{figure}[t!]%
	
	\definecolor{uuuuuu}{rgb}{0.26,0.26,0.26}
	\definecolor{ffffff}{rgb}{1.,1.,1.}
	\definecolor{eqeqeq}{rgb}{0.87,0.87,0.87}
	
	\begin{subfigure}[$\cI(r,\beta)$\label{fig:caso1inner}]
		{\begin{tikzpicture}[line cap=round,line join=round,>=triangle 45,x=0.2cm,y=0.2cm]
			\clip(-12.2,-20) rectangle (12.2,10.7);
			\fill[line width=0.pt,dotted,color=eqeqeq,fill=eqeqeq,fill opacity=1.0] (-6.,0.) -- (-6.,-14.873879815046912) -- (5.999999761581421,-14.87387981504691) -- (6.,0.) -- cycle;
			\fill[dotted,color=ffffff,fill=ffffff,fill opacity=1.0] (-6.,0.) -- (0.,-13.610007375180219) -- (6.,0.) -- cycle;
			\draw [shift={(-6.,0.)}] plot[domain=4.71238898038469:5.127609500981089,variable=\t]({1.*14.873879815046912*cos(\t r)+0.*14.873879815046912*sin(\t r)},{0.*14.873879815046912*cos(\t r)+1.*14.873879815046912*sin(\t r)});
			\draw [shift={(6.,0.)}] plot[domain=4.71238898038469:5.127609500981089,variable=\t]({-1.*14.873879815046912*cos(\t r)+0.*14.873879815046912*sin(\t r)},{0.*14.873879815046912*cos(\t r)+1.*14.873879815046912*sin(\t r)});
			\draw (-12.,0.)-- (-6.,0.)-- (0.,10.392304845413266)-- (6.,0.)-- (12.,0.);
			\draw [shift={(-6.,0.)},line width=0.4pt,dotted,color=ffffff,fill=ffffff,fill opacity=1.0]  (0,0) --  plot[domain=4.71238898038469:5.127609500981089,variable=\t]({1.*14.873879815046912*cos(\t r)+0.*14.873879815046912*sin(\t r)},{0.*14.873879815046912*cos(\t r)+1.*14.873879815046912*sin(\t r)}) -- cycle ;
			\draw [line width=0.4pt,dotted,color=ffffff] (-6.,0.)-- (0.,-13.610007375180219);
			\draw [line width=0.4pt,dotted,color=ffffff] (0.,-13.610007375180219)-- (6.,0.);
			\draw [line width=0.4pt,dotted,color=ffffff] (6.,0.)-- (-6.,0.);
			\draw [shift={(6.,0.)},color=ffffff,fill=ffffff,fill opacity=1.0]  (0,0) --  plot[domain=4.297168459788289:4.712388964355343,variable=\t]({1.*14.873879815046912*cos(\t r)+0.*14.873879815046912*sin(\t r)},{0.*14.873879815046912*cos(\t r)+1.*14.873879815046912*sin(\t r)}) -- cycle ;
			\draw [shift={(-6.,0.)}] plot[domain=4.71238898038469:5.127609500981089,variable=\t]({1.*14.873879815046912*cos(\t r)+0.*14.873879815046912*sin(\t r)},{0.*14.873879815046912*cos(\t r)+1.*14.873879815046912*sin(\t r)});
			\draw [shift={(6.,0.)}] plot[domain=4.297168459788289:4.712388964355343,variable=\t]({1.*14.873879815046912*cos(\t r)+0.*14.873879815046912*sin(\t r)},{0.*14.873879815046912*cos(\t r)+1.*14.873879815046912*sin(\t r)});
			\draw (-11.99069719031515,-14.873879815046909)-- (12.009302809684852,-14.873879815046909);
			\draw [line width=0.4pt,dotted] (-6.,0.)-- (-6.,-14.873879815046912);
			\draw [dotted] (-6.,0.)-- (6.,0.);
			\draw [line width=0.4pt,dotted] (-6.,0.)-- (0.,-13.610007375180219);
			\draw [line width=0.4pt,dotted] (6.,0.)-- (0.,-13.610007375180219);
			\draw [line width=0.4pt,dotted] (6.,0.)-- (6.,-14.873879815046912);
			\begin{scriptsize}
			\draw (0,1) node {$\beta$};
			\draw (-3.5,-4.5) node[anchor=north west] {$r$};
			\draw (2.939355229539017,7.394084842414974) node[anchor=north west] {$\partial K_n$};
			\end{scriptsize}
			\end{tikzpicture}}
	\end{subfigure}
	~
	\begin{subfigure}[$A(x)$\label{fig:caso2inner}]
		{\begin{tikzpicture}[line cap=round,line join=round,>=triangle 45,x=0.2cm,y=0.2cm]
			\clip(-12.2,-20.) rectangle (12.2,10.7);
			\fill[line width=0.pt,dotted,color=eqeqeq,fill=eqeqeq,fill opacity=1.0] (-6.,0.) -- (-6.,-14.873879815046912) -- (5.999999761581421,-14.87387981504691) -- (6.,0.) -- cycle;
			\fill[dotted,color=ffffff,fill=ffffff,fill opacity=1.0] (-6.,0.) -- (0.,-13.610007375180219) -- (6.,0.) -- cycle;
			\fill[line width=0.pt,color=ffffff,fill=ffffff,fill opacity=1.0] (-6.,-11.547080338414503) -- (-6.,-14.873879815046912) -- (-0.23781427987911857,-14.873879815046909) -- cycle;
			\draw [shift={(-6.,0.)}] plot[domain=4.71238898038469:5.127609500981089,variable=\t]({1.*14.873879815046912*cos(\t r)+0.*14.873879815046912*sin(\t r)},{0.*14.873879815046912*cos(\t r)+1.*14.873879815046912*sin(\t r)});
			\draw [shift={(6.,0.)}] plot[domain=4.71238898038469:5.127609500981089,variable=\t]({-1.*14.873879815046912*cos(\t r)+0.*14.873879815046912*sin(\t r)},{0.*14.873879815046912*cos(\t r)+1.*14.873879815046912*sin(\t r)});
			\draw (-12.,0.)-- (-6.,0.)-- (0.,10.392304845413266)-- (6.,0.)-- (12.,0.);
			\draw [shift={(-6.,0.)},line width=0.4pt,dotted,color=ffffff,fill=ffffff,fill opacity=1.0]  (0,0) --  plot[domain=4.71238898038469:5.127609500981089,variable=\t]({1.*14.873879815046912*cos(\t r)+0.*14.873879815046912*sin(\t r)},{0.*14.873879815046912*cos(\t r)+1.*14.873879815046912*sin(\t r)}) -- cycle ;
			\draw [line width=0.4pt,dotted,color=ffffff] (-6.,0.)-- (0.,-13.610007375180219);
			\draw [line width=0.4pt,dotted,color=ffffff] (0.,-13.610007375180219)-- (6.,0.);
			\draw [line width=0.4pt,dotted,color=ffffff] (6.,0.)-- (-6.,0.);
			\draw [shift={(6.,0.)},color=ffffff,fill=ffffff,fill opacity=1.0]  (0,0) --  plot[domain=4.297168459788289:4.712388964355343,variable=\t]({1.*14.873879815046912*cos(\t r)+0.*14.873879815046912*sin(\t r)},{0.*14.873879815046912*cos(\t r)+1.*14.873879815046912*sin(\t r)}) -- cycle ;
			\draw [shift={(-6.,0.)}] plot[domain=4.71238898038469:5.127609500981089,variable=\t]({1.*14.873879815046912*cos(\t r)+0.*14.873879815046912*sin(\t r)},{0.*14.873879815046912*cos(\t r)+1.*14.873879815046912*sin(\t r)});
			\draw [shift={(6.,0.)}] plot[domain=4.297168459788289:4.712388964355343,variable=\t]({1.*14.873879815046912*cos(\t r)+0.*14.873879815046912*sin(\t r)},{0.*14.873879815046912*cos(\t r)+1.*14.873879815046912*sin(\t r)});
			\draw (-11.99069719031515,-14.873879815046909)-- (12.009302809684852,-14.873879815046909);
			\draw [line width=0.4pt,dotted] (-6.,0.)-- (-6.,-14.873879815046912);
			\draw [dotted] (-6.,0.)-- (6.,0.);
			\draw [line width=0.4pt,dotted] (-6.,0.)-- (0.,-13.610007375180219);
			\draw [line width=0.4pt,dotted] (6.,0.)-- (0.,-13.610007375180219);
			\draw [line width=0.4pt,dotted] (6.,0.)-- (6.,-14.873879815046912);
			\draw [line width=0.4pt,dashed] (-14.747666163389441,-6.496612924200624)-- (12.442250905395332,-22.19471886310712);
			\draw [line width=0.4pt,dotted] (-1.4646639856000565,0.)-- (-1.4646639856000565,-14.165557807193714);
			\begin{scriptsize}
			\draw (-3.5,-4.5) node[anchor=north west] {$r$};
			\draw (2.960871488977172,7.375116067609024) node[anchor=north west] {$\partial K_n$};
			\draw[color=black] (-6.937492577479757,-10) node {$\ell$};
			\draw[color=ffffff] (-2.5198456074858533,-13.862662681184919) node {$e$};
			\draw[color=ffffff] (3.338851978388415,-13.862662681184919) node {$p$};
			\draw[color=black] (-2.4,-12.7) node {$Q_x$};
			\draw[color=uuuuuu] (-1.3622803586639214,0.9966953087943806) node {$P_x$};
			\draw[color=uuuuuu] (-6.6,1) node {$P_\iota$};			\end{scriptsize}
			\end{tikzpicture}
		}
	\end{subfigure}
	~
	\begin{subfigure}[$B(x)$\label{fig:caso3inner}]{\begin{tikzpicture}[line cap=round,line join=round,>=triangle 45,x=0.2cm,y=0.2cm]
			\clip(-12.2,-20.) rectangle (12.2,10.7);
			\fill[line width=0.pt,dotted,color=eqeqeq,fill=eqeqeq,fill opacity=1.0] (-6.,0.) -- (-6.,-14.873879815046912) -- (5.999999761581421,-14.87387981504691) -- (6.,0.) -- cycle;
			\fill[dotted,color=ffffff,fill=ffffff,fill opacity=1.0] (-6.,0.) -- (0.,-13.610007375180219) -- (6.,0.) -- cycle;
			\fill[line width=0.pt,color=ffffff,fill=ffffff,fill opacity=1.0] (-6.,-9.706855040461368) -- (-6.,-14.873879815046912) -- (2.949549433549285,-14.87387981504691) -- cycle;
			\draw [shift={(-6.,0.)}] plot[domain=4.71238898038469:5.127609500981089,variable=\t]({1.*14.873879815046912*cos(\t r)+0.*14.873879815046912*sin(\t r)},{0.*14.873879815046912*cos(\t r)+1.*14.873879815046912*sin(\t r)});
			\draw [shift={(6.,0.)}] plot[domain=4.71238898038469:5.127609500981089,variable=\t]({-1.*14.873879815046912*cos(\t r)+0.*14.873879815046912*sin(\t r)},{0.*14.873879815046912*cos(\t r)+1.*14.873879815046912*sin(\t r)});
			\draw (-12.,0.)-- (-6.,0.)-- (0.,10.392304845413266)-- (6.,0.)-- (12.,0.);
			\draw [shift={(-6.,0.)},line width=0.4pt,dotted,color=ffffff,fill=ffffff,fill opacity=1.0]  (0,0) --  plot[domain=4.71238898038469:5.127609500981089,variable=\t]({1.*14.873879815046912*cos(\t r)+0.*14.873879815046912*sin(\t r)},{0.*14.873879815046912*cos(\t r)+1.*14.873879815046912*sin(\t r)}) -- cycle ;
			\draw [line width=0.4pt,dotted,color=ffffff] (-6.,0.)-- (0.,-13.610007375180219);
			\draw [line width=0.4pt,dotted,color=ffffff] (0.,-13.610007375180219)-- (6.,0.);
			\draw [line width=0.4pt,dotted,color=ffffff] (6.,0.)-- (-6.,0.);
			\draw [shift={(6.,0.)},color=ffffff,fill=ffffff,fill opacity=1.0]  (0,0) --  plot[domain=4.297168459788289:4.712388964355343,variable=\t]({1.*14.873879815046912*cos(\t r)+0.*14.873879815046912*sin(\t r)},{0.*14.873879815046912*cos(\t r)+1.*14.873879815046912*sin(\t r)}) -- cycle ;
			\draw [shift={(-6.,0.)}] plot[domain=4.71238898038469:5.127609500981089,variable=\t]({1.*14.873879815046912*cos(\t r)+0.*14.873879815046912*sin(\t r)},{0.*14.873879815046912*cos(\t r)+1.*14.873879815046912*sin(\t r)});
			\draw [shift={(6.,0.)}] plot[domain=4.297168459788289:4.712388964355343,variable=\t]({1.*14.873879815046912*cos(\t r)+0.*14.873879815046912*sin(\t r)},{0.*14.873879815046912*cos(\t r)+1.*14.873879815046912*sin(\t r)});
			\draw (-11.99069719031515,-14.873879815046909)-- (12.009302809684852,-14.873879815046909);
			\draw [line width=0.4pt,dotted] (-6.,0.)-- (-6.,-14.873879815046912);
			\draw [dotted] (-6.,0.)-- (6.,0.);
			\draw [line width=0.4pt,dotted] (-6.,0.)-- (0.,-13.610007375180219);
			\draw [line width=0.4pt,dotted] (6.,0.)-- (0.,-13.610007375180219);
			\draw [line width=0.4pt,dotted] (6.,0.)-- (6.,-14.873879815046912);
			\draw [line width=0.4pt,dashed] (-11.402169580306287,-6.587910979063524)-- (11.396274686895758,-19.750598913837308);
			\draw [line width=0.4pt,dotted] (2.000918488815551,0.)-- (2.000918488815551,-14.326187483743281);
			\begin{scriptsize}
			\draw (-3.5,-4.5) node[anchor=north west] {$r$};
			\draw (2.9311857130840737,7.373735943162981) node[anchor=north west] {$\partial K_n$};
			\draw[color=black] (-6.937492577479757,-8.4) node {$\ell$};
			\draw[color=ffffff] (-2.544105629760416,-13.854618203882152) node {$e$};
			\draw[color=ffffff] (3.3487926799111962,-13.854618203882152) node {$p$};
			\draw[color=black] (3.2,-13.5) node {$Q_x$};
			\draw[color=uuuuuu] (2.4439775851190984,1.0632306666643179) node {$P_x$};
			\draw[color=uuuuuu] (-6.6,1) node {$P_\iota$};
			\end{scriptsize}
			\end{tikzpicture}}
	\end{subfigure}
	\caption{The shapes occurring in the decomposition of  $K_n^{r_n}\setminus K_2^{r_n}$ are shown in gray color (with the corresponding areas indicated below).}
	\label{fig:InnerCasi23}
\end{figure}

Then, for $n\ge 5$, the area of the inner Cheeger $K^r_n$ is given by
\begin{equation*}
|K^r_n| =  (6\sqrt{3} -\pi)r^2 -12 r +3\sqrt{3} +12\left[ \sum_{j=5}^n c(x,j)\mathcal{I}(r, 3^{2-j}) + \phi(x) \right]\,,
\end{equation*}
where $c(x, j)$ denotes the number of triangles added at step $j$ of the construction of $K_n$ with a side contained in the segment $\overline{P_xW}$ (see Figure~\ref{fig:QxPx}) and 
\[
\phi(x) =
\begin{cases}
0\, \quad &\text{if $d =0$}\,,\\
A(x)\, &\text{if $d=x-\iota >0$ }\,,\\
B(x)\, &\text{if $d=\beta+\iota-x>0$}\,.
\end{cases}
\]
Thus, one needs to solve, with respect to $x$
\begin{equation}\label{eq:tosolven}
f_n(x) := |K^r_n| - \pi r^2 = 0\,,
\end{equation}
with the constraint $x\in [x_{2}, 1]$. If one knows $x_{n-1}$, then the constraint becomes more stringent, namely $x\in [x_{n-1}, 1]$. Notice that the function $f_n$ is Lipschitz on $[x_2,1]$ and strictly decreasing. Moreover $f_n(1)<0<f_n(x_2)$, hence \eqref{eq:tosolven} has a unique solution $x_n \in (x_2,1)$. Furthermore, the function relating $x$ and $r$, which we recall to be $r = \sqrt{x^2/3 +d^2}$, is Lipschitz for $x\in [x_2, 1]$ and its almost-everywhere defined derivative is bounded from below by a positive constant. Thus, this relation is invertible and to each $x_n$ corresponds a unique $r_n$.

We now compute $h(K_5)$. Observing that 
\[
f_5(r_2\sqrt{3}) = 12 \cI(r_2, 1/27) > 0
\]
and
\[
f_5(25/27) = \frac{2211\sqrt{3} -1250\pi}{(27\sqrt{3})^2} +12\cI\left(\frac{25}{27\sqrt{3}}, 1/27 \right)<0 \,,
\]
we deduce that the solution must belong to the interval $[r_2\sqrt{3},25/27]$, hence $d_5 = 0$ and the equation to be solved (here depending only on the variable $r$) reduces to
\begin{equation}\label{eq:K5}
(6\sqrt{3} -2\pi)r^2 -12 r +3\sqrt{3} +12 \cI(r,1/27) =0\,.
\end{equation}
By numerically solving \eqref{eq:K5}, we obtain $r_5 = 0.528751827\dots$ 

Now, since Theorem \ref{theorem:Cheeger formula} holds, we can apply Lemma \ref{lem:approx} to deduce an estimate for $h(K)$. We note that, for $n\ge 4$, the area of the difference set $K^{r_n}\setminus K_n^{r_n}$ is estimated by
\begin{equation}\label{eq:stima2K}
\left|K^{r_n}\setminus K_n^{r_n}\right| \le 12\sum_{j=n+1}^{\infty} 2^{j-5}\cI(r_n,3^{2-j})\,.
\end{equation}
Then, by \eqref{eq:stimaIrbeta}, one obtains
\[
\mathcal{I}(r_n, 3^{2-j}) \le \frac{1}{12 r_n 3^{3j-6}}\,.
\]
Plugging this last inequality into \eqref{eq:stima2K} gives
\begin{equation*}\label{eq:stima2Kbis}
\left|K^{r_n}\setminus K_n^{r_n}\right| \le \frac{12}{r_n}\sum_{j=n+1}^{\infty} \frac{2^{j-5}}{12\cdot 3^{3j-6}}\, \le \frac{1}{r_n}\sum_{j=n+1}^{\infty} \frac{2^{j-5}}{3^{3j-6}}\,.
\end{equation*}
Applying Lemma \ref{lem:approx} with $\Om = K$ and $\om = K_n$, recalling that $1/r_n \le 2$, and setting $\bar r = h(K)^{-1}$, we obtain the error estimate 
\begin{equation}\label{eq:stimafinale}
0 \le \bar r- r_n \le  \frac{1}{2\pi r_n^2}\sum_{j=n+1}^{\infty} \frac{2^{j-5}}{3^{3j-6}} \le \sum_{j=n+1}^{\infty} \frac{2^{j-4}}{3^{3j-5}} = \frac{2^{n-3}}{25\cdot 3^{3n-5}}\,.
\end{equation}
Computing \eqref{eq:stimafinale} for the particular value $n=5$ reveals that $r_5$ approximates $\bar r$ with a precision of $10^{-5}$ (indeed, the error is less than $3\cdot 10^{-6}$). Coupling this with $r_5\ge 1/2$ finally gives
\[
|h(K) - h(K_5)| < 12 \cdot 10^{-6} < 10^{-4}  
\]
and thus
\[
h(K)= 1.8912\dots
\]
Thanks to similar, but more tedious, computations for the cases $n=6,7$ and $8$, one shows that \eqref{eq:tosolven} has a solution $x_n$ for which $\phi(x_n)=0$. In particular, this gives
\[
h(K) = 1.89124548\dots
\]
where the displayed digits are exact.

\bibliographystyle{alpha}
\bibliography{CheegerFinal}

\end{document}